\documentclass[reqno]{amsart}
\usepackage{stmaryrd}
\usepackage{amsmath, amssymb, amsthm, epsfig}
\usepackage{hyperref, latexsym}
\usepackage{url}
\usepackage[mathscr]{euscript}

\usepackage{color}
\usepackage{fullpage} 
\usepackage{setspace}

\def\today{\ifcase\month\or
  January\or February\or March\or April\or May\or June\or
  July\or August\or September\or October\or November\or December\fi
  \space\number\day, \number\year}

\DeclareMathOperator{\Si}{\mathrm{Si}}

 \newtheorem{theorem}{Theorem}[subsection]
 
  \newtheorem{conjecture}[theorem]{Conjecture}
 \newtheorem{lemma}[theorem]{Lemma}
 \newtheorem{proposition}[theorem]{Proposition}
 \newtheorem{corollary}[theorem]{Corollary}
 \theoremstyle{definition}

 \theoremstyle{remark}

 \newcommand{\R}{\mathbb{R}}
 \newcommand{\N}{\mathbb{N}}
 
 \newcommand{\Z}{\mathbb{Z}}

  \renewcommand{\d}{\text{\rm d}}
 \newcommand{\du}{\text{\rm d}u}

\newcommand{\im}{{\rm Im}\,}
\newcommand{\re}{{\rm Re}\,}

\renewcommand{\Re}{\text{Re}}

\begin{document}
\title[On the number variance of zeta zeros]{On the number variance of zeta zeros \\ and a conjecture of Berry}
\author[Lugar]{Meghann Moriah Lugar}
\author[Milinovich]{Micah B. Milinovich}
\author[Quesada-Herrera]{Emily Quesada-Herrera}
\address{University of Mississippi - University, MS 38677 USA}
\email{mmgibso1@olemiss.edu}
\address{University of Mississippi - University, MS 38677 USA}
\email{mbmilino@olemiss.edu}
\address{Graz University of Technology, Institute of Analysis and Number Theory, Steyrergasse 30/II, 8010 Graz, Austria}
\email{quesada@math.tugraz.at}

\allowdisplaybreaks
\numberwithin{equation}{section}

\maketitle

\begin{abstract}
Assuming the Riemann hypothesis, we prove estimates for the variance of the real and imaginary part of the logarithm of the Riemann zeta-function in short intervals. We give three different formulations of these results.  Assuming a conjecture of Chan for how often gaps between zeros can be close to a fixed nonzero value, we prove a conjecture of Berry (1988) for the number variance of zeta zeros in the non-universal regime. In this range, GUE statistics do not describe the distribution of the zeros.  We also calculate lower-order terms in the second moment of the logarithm of the modulus of the Riemann zeta-function on the critical line. Assuming Montgomery's pair correlation conjecture, this establishes a special case of a conjecture of Keating and Snaith (2000). 
\end{abstract}

\section{Introduction}
Understanding the distribution of the zeros of the Riemann zeta-function, $\zeta(s)$, is an important problem in number theory. Let $N(t)$ be the number of zeros $\rho=\beta+i\gamma$ of $\zeta(s)$ such that $0<\gamma \le t$ and $0< \beta < 1$ (counted with multiplicity, where the zeros with $\gamma=t$ are counted with weight $\frac{1}{2}$). It is known that 
\begin{equation}\label{N(T)definition}
    N(t)= \frac{t}{2\pi}\log \frac{t}{2\pi} -\frac{t}{2\pi} +\frac{7}{8} +S(t) + O\!\left(\frac{1}{t} \right), 
\end{equation}
where, for $t\neq \gamma$, the function is defined by
\begin{equation*}
    S(t)=\frac{1}{\pi}\arg \zeta\!\left(\tfrac{1}{2}+it\right)
\end{equation*}
with the argument obtained by a continuous variation along the straight line segments joining the points $2,$ $2+it$, and $\frac{1}{2}+it$ starting with the value $\arg \zeta(2)=0.$ If $t=\gamma$ for a zero of $\zeta(s)$, we define
\begin{equation*}
    S(t)=\lim_{\varepsilon\to0}\frac{S(t+\varepsilon)+S(t-\varepsilon)}{2}.
\end{equation*}
By equation \eqref{N(T)definition} and the well-known estimates $S(t)\ll \log t$ and $\int\limits_0^T S(t)\,\d t \ll  \log T$, we can think of $S(t)$ as the difference between the actual and average number of zeros around height $t$. 

\smallskip

From \eqref{N(T)definition}, we expect that there are about $\delta$ zeros of $\zeta(s)$ with ordinates in the interval $[t, t+\frac{2\pi\delta}{\log T}]$ when $0<t \le T$ and $T$ is large. We define the {\it number variance} of the zeros of $\zeta(s)$ by
\begin{equation}\label{eq:number-variance}
    \int_0^T \left[N\!\left(t+\tfrac{2\pi\delta}{\log T}\right)-N(t)-\delta\right]^{2} \d t.
\end{equation}
This quantity has been studied by a number of authors, for instance \cite{Be, BeK, Fu0, Fu1, Fu3, GaMu}.
By \eqref{N(T)definition}, up to a small error, the integral in \eqref{eq:number-variance} is equal to
\[
\int_0^T \left[S\!\left(t+\tfrac{2\pi\delta}{\log T}\right)-S(t)\right]^{2} \d t.
\]
As described in Section \ref{BS}, Berry \cite{Be} (see also \cite{BeK}) has given a precise conjecture for the asymptotic behavior for this integral. In the \textit{universal regime} of his model, when $\delta=o(\log T)$, Berry conjectured an asymptotic formula that matches exactly the variance of eigenvalues of GUE random matrices. However, when $\delta \gg \log T$, in the so-called \textit{non-universal regime} of his model, his conjecture is no longer described by the predictions from GUE and incorporates additional input from the primes. 

\smallskip

Building upon ideas of Selberg \cite{SB2} and Goldston \cite{Go}, Gallagher and Mueller \cite{GaMu} and Fujii \cite{Fu1} have given a conditional proof of Berry's conjecture in the universal regime assuming both the Riemann Hypothesis (RH) and versions of Montgomery's pair correlation conjecture. In this paper, we introduce new ideas to prove novel results on the number variance of zeta zeros in the non-universal regime when $\delta \gg \log T$. In particular, we show that new input from both the zeros and primes is needed in this regime, requiring information on the zeros beyond pair correlation (since we no longer expect GUE behavior in this range). In Section \ref{NV}, we give three different formulations of these results, stated as Theorems \ref{thm:short-intervals-keating} -- \ref{thm:short-intervals-berry}. In Section \ref{BS}, we show how our results give a conditional proof of Berry's conjecture in the non-universal regime assuming RH and a conjecture of Chan \cite{Ch} for the pair correlation of zeta zeros in longer ranges (which examines how often normalized gaps between zeros can be close to a fixed nonzero value).  
Roughly, pair correlation studies the distribution of gap sizes localized near zero with respect to the average spacing, whereas our new results require information about the distribution of gap sizes localized near other points.

\smallskip

Before stating our new results 
on the number variance of zeta zeros, we first describe the work of Selberg \cite{SB2,SB3} and Goldston \cite{Go} on the moments of $S(t)$ and $\log|\zeta(\frac{1}{2}+it)|$ and the connection to the pair correlation of zeta zeros.
Analogous to a result of Goldston for $S(t)$, our Theorem \ref{thm:variance} gives lower-order terms for the second moment of $\log|\zeta(\frac{1}{2}+it)|$ assuming RH, in terms of the pair correlation of zeta zeros. Assuming Montgomery's pair correlation conjecture, Theorem \ref{thm:variance} establishes a special case of a conjecture of Keating and Snaith \cite{KS}.

\subsection{Selberg's Central Limit Theorem} 
A celebrated and classical result of Selberg is that the real and imaginary parts of the logarithm of the Riemann zeta-function are normally distributed on the critical line. He proved this by estimating moments of $S(t)$, first assuming RH and then later without any conditions with the same main term and a slightly weaker error term \cite{SB2,SB3}. Assuming RH, Selberg showed that, for $k\in \N$ and $T\ge 3$, 
\begin{equation}\label{SelbergS(T)moments} 
\int\limits_0^T S(t)^{2k} \, \d t = \frac{(2k)!}{k! (2\pi)^{2k}} T (\log\log T)^k \left[ 1 + O\!\left( \frac{1}{\log\log T} \right)\right].
\end{equation}
In other words, the moments of $S(t)$ are Gaussian. In this way, Selberg \cite{SB4} deduces a \textit{central limit theorem} for $S(t)$:
\begin{equation}\label{S(T)centrallimittheorem}
\lim \limits_{T\rightarrow \infty} \frac{1}{T} \text{ meas}\left\lbrace T\leq t\leq 2T : \frac{\pi S(t)}{\sqrt{\frac12 \log \log T}} \in [a,b] \right\rbrace= \frac{1}{\sqrt{\pi}} \int \limits_{a}^b e^{-x^2/2}\, \d x.
\end{equation}
This tells us $\pi S(t)$ is normally distributed for $t \in[T,2T]$ with mean 0 and variance $\tfrac12 \log\log T$, when $T$ is large. Selberg (unpublished) also considered the moments of $\log|\zeta(\frac{1}{2}+it)|$. Using techniques outlined by Tsang \cite{TS1}, assuming RH it can be shown that
\begin{equation}\label{SelbergLogZetamoments}
\int \limits_0^T\log^{2k}\!|\zeta(\tfrac12 +it)| \, \d t = \frac{(2k)!}{k!2^{2k}} T (\log\log T)^k \left[ 1 + O\!\left( \frac{1}{\log\log T} \right)\right].
\end{equation}
These moments can be calculated unconditionally with a slightly weaker error term. A corresponding central limit theorem for $\log|\zeta(\tfrac12 +it)|$, analogous to \eqref{S(T)centrallimittheorem}, follows from the work Selberg and Tsang. See Radziwi\l\l\, and Soundararajan  \cite{RaSo} for a recent and simplified proof of Selberg's central limit theorem for $\log|\zeta(\tfrac12 +it)|$.

\subsection{The variance in Selberg's central limit theorem}

Selberg modeled $\log\zeta(s)$ near the critical line using information from the primes and the zeros of $\zeta(s)$. He arrives at the main term in \eqref{SelbergS(T)moments} using information from the primes. The information about the zeros is cleverly contained in his error term. \smallskip

Recall that the {\it variance} of a distribution is given by its second moment, which corresponds to taking $k=1$ in \eqref{SelbergS(T)moments}.  
Goldston \cite{Go} gave a refined estimate for the variance of $S(t)$ in Selberg's central limit theorem utilizing finer information from both the primes and the zeros of $\zeta(s)$ in his representation of $\log \zeta(s)$. He does so through methods relying, in part, on Montgomery's work \cite{Mo} on the pair correlation of the zeros of $\zeta(s)$. Assuming RH, Goldston shows that
\begin{equation*}\label{GoldstonR2}
\int\limits_0^T |S(t)|^2 \, \d t = \frac{T}{2\pi^2} \log\log T +\frac{aT}{\pi^2} + o(T),
\end{equation*}
as $T\to\infty$, where the constant $a$ is given by
\begin{equation}\label{GoldstonConstant}
a= \frac{1}{2} \left( \gamma_0 +  \sum\limits_{m=2}^{\infty} \sum \limits_{p}\left( \frac{1}{m^2}-\frac{1}{m}\right) \frac{1}{p^m} + \int\limits_{1}^{\infty} \frac{F(\alpha)}{\alpha^2} \, \d \alpha\right),
\end{equation}
$\gamma_0$ is Euler's constant, and the sum over $p$ runs over the primes.
Here, the term with $F(\alpha)$ captures the information from the zeros of $\zeta(s)$. The function $F(\alpha)$ was introduced by Montgomery \cite{Mo} to study the pair correlation of the zeros of $\zeta(s)$,  and is defined by
\begin{equation}  \label{FFunctionAlphaT}
F(\alpha)=F(\alpha,\,T):=\bigg(\dfrac{T}{2\pi}\log T\bigg)^{-1}\displaystyle\sum_{0<\gamma, \gamma'\leq T}T^{\,i\alpha(\gamma-\gamma')}w(\gamma-\gamma'),
\end{equation}
for $\alpha\in \R$ and $T\geq 2$, where $w(u)=4/(4+u^2)$, and the double sum runs over the ordinates, $\gamma$ and $\gamma'$, of non-trivial zeros of $\zeta(s)$.  In this way, we see that Goldston's result contains information from both the primes and the zeros in the definition of the constant $a$. As initially defined, the constant $a$ actually depends on $T$. In Lemma \ref{lem:a(T)} we show that this dependence is mild (see also \cite[Theorem 2]{Go}).\smallskip

Our first theorem is an analogue of Goldston's more precise result for the second moment of $\log |\zeta(\tfrac{1}{2}+it)|$, refining the case $k=1$ in \eqref{SelbergLogZetamoments}.

\begin{theorem}\label{thm:variance}
Assume RH and let $F(\alpha)$ be defined by \eqref{FFunctionAlphaT}. Then, as $T\to\infty$, 
\begin{equation*}
\int\limits_0^T \log^2 |\zeta(\tfrac{1}{2}+it)| \,\d t= \frac{T}{2} \log\log T +a T + o(T),
\end{equation*}
where the constant $a$ is given by \eqref{GoldstonConstant}.
\end{theorem}

Montgomery's Strong Pair Correlation Conjecture states that $F(\alpha)\sim 1$ uniformly on compact intervals when $\alpha >1$ and $T\to\infty$. Assuming this conjecture, we see that
\[
a= \frac{1}{2} \left( 1 + \gamma_0 +  \sum\limits_{m=2}^{\infty} \sum \limits_{p}\left( \frac{1}{m^2}-\frac{1}{m}\right) \frac{1}{p^m} \right),
\]
and Theorem \ref{thm:variance} establishes a special case of a conjecture of Keating and Snaith \cite[eq.~ (97)]{KS}, who conjectured a formula for all moments. They made analogous conjectures for families of $L$-functions in \cite{KS2}.

\smallskip

Though the statement of our first result is very similar to Goldston's theorem, the proofs are considerably different. One reason for this is easy to explain. From the formula for $N(t)$ in \eqref{N(T)definition}, we see that the function $S(t)$ is bounded near the zeros of $\zeta(\frac{1}{2}+it)$, with a jump discontinuity at each zero. On the other hand, $\log|\zeta(\frac{1}{2}+it)|$ is not bounded near the zeros, and can be arbitrarily large in the negative direction. These logarithmic singularities do not substantially change the end result, but they do cause technical difficulties within the proof. Another major difference from Goldston's work is that our proof relies on a delicate cancellation of main terms, which we accomplish by introducing the function $g(x)$ in Section \ref{sec:representation}. Though an analogous cancelation of main terms was not present in Goldston's work, it is present in the work of Chirre and the third author who introduced related functions to obtain similar cancellations in the study of the second moment of the iterated antiderivatives of $S(t)$ (see \cite[Lemma 4]{CQ-H2}). However, as we shall see, when considering $\log|\zeta(\frac{1}{2}+it)|$ there are important new technical differences in the properties of our functions due to the unbounded discontinuities.  

\smallskip

{\sc Remark}. The error term in Theorem \ref{thm:variance} could be improved using additional assumptions such as a quantitative form of the twin prime conjecture (see \cite{Ch2}), or a more precise conjectural formula for pair correlation by Bogomolny and Keating \cite{BK} or Conrey and Snaith \cite{CoSn}. For details, see the work of Chan \cite{Ch1}. 

\subsection{Number variance of zeta zeros}\label{NV} In a series of papers, Fujii \cite{Fu0,Fu1,Fu3} considered the $2k$th moments of the difference $S\! \left(t+\Delta\right)-S(t)$.  Using Selberg's methods, for $T$ sufficiently large, Fujii \cite{Fu0} showed unconditionally that 
\begin{equation}\label{FujiiMoments1}
\begin{split}
\int \limits_0^T &\left[S\!\left(t+\Delta\right)-S(t)\right]^{2k} \d t  \, 
=\,\frac{(2k)!}{(2\pi)^{2k}k!} T(2\log(2 + \Delta\log T))^k + O\!\left( T \, (\log(2 + \Delta\log T))^{k-1/2}\right)
\end{split}
\end{equation}
when $0< \Delta \ll 1$ and, assuming RH, Fujii \cite{Fu3} showed that 
\begin{equation}\label{FujiiMoments2}
\begin{split}
\int \limits_0^T \left[S\!\left(t+\Delta\right)-S(t)\right]^{2k} \d t  \, 
&=\,\frac{(2k)! 2^k}{(2\pi)^{2k}k!} T \big( \log\log T - \log|\zeta(1+i\Delta)| \big)^k + O\Big( T \big( \log\log T \big)^{k-1} \Big)
\end{split}
\end{equation}
when $1 \le \Delta \le T$. Fujii's result in \eqref{FujiiMoments1} gives an asymptotic formula when $\Delta \log T$ goes to infinity with $T$ (sufficiently slowly). 
 If $\Delta \log T \ll 1$, then the main term and error term in this result are the same order of magnitude and this result does not give an asymptotic formula. Fujii's result in \eqref{FujiiMoments2} has an error term of the same order of magnitude as Selberg's conditional result in \eqref{SelbergS(T)moments}. In particular, when $k=1$, the error term in \eqref{FujiiMoments2} is $O(T)$. In this case, similarly to Goldston's result for the second moment of $S(t)$, the contribution from the zeros of $\zeta(s)$ gives a term of size $T$, and Selberg's method cannot be used be used to identify this main term. Realizing this, Fujii \cite{Fu1} applies Goldston's methods  \cite{Go} to his own work and, assuming RH, he shows that for $0<\Delta = o(1)$ we have
\begin{equation}\label{FujiiResult}
   \!\int\limits_0^T \!\left[S\!\left(t+\Delta\right)-S(t)\right]^2 \d t =
    \frac{T}{ \pi^2}\left\{
    \int\limits_{0}^{1} \frac{1-\cos (\alpha\Delta\log T )}{\alpha}\, \d \alpha 
    +
    \int\limits_{1}^{\infty} \frac{F(\alpha) \, [1-\cos(\alpha\Delta\log T)]}{\alpha^2}\, \d \alpha \right\} + o(T),
\end{equation}
as $T\to\infty$. Gallagher and Mueller \cite{GaMu} had previously given a similar estimate in the limited range $\Delta \asymp \tfrac{1}{\log T}$ assuming both RH and Montgomery's pair correlation conjecture. A calculation related to \eqref{FujiiResult} can also be found in recent work of Heap \cite[Proposition 9]{WH}. Notice that the expression of the main term in \eqref{FujiiResult} is stated using information from the zeros, in the form of $F(\alpha)$. As we shall see, more information about the distribution of the zeros of $\zeta(s)$ is required in order to  accurately describe the situation when $\Delta \gg 1$.

\smallskip

Our next results 
refine Fujii's calculation in \eqref{FujiiResult} by giving an asymptotic formula of similar precision but with a much larger range of $\Delta$. This requires expressing the main term in a different manner, giving a better understanding of the behavior of the number variance for zeta zeros for different sizes of $\Delta$. To achieve this, we must overcome significant technical challenges, as new main terms arise and a more careful consideration of the error terms is required. Our result relies on finer information from both the primes and the zeros of $\zeta(s)$. In particular, we require a variation of Montgomery's function $F(\alpha)$ introduced by Chan \cite{Ch} in his study of the pair correlation of zeros in {\it longer ranges}. We define\footnote{\,\,\,Chan works only with the real part of $F_\Delta$.}
\begin{equation}\label{eq:F_delta_def}
    F_\Delta(\alpha)=F_\Delta(\alpha,\,T):= \frac{2\pi}{T\log T}\sum_{0<\gamma, \, \gamma'\le T}T^{i\alpha\left(\gamma-\gamma'-\Delta\right)} \, w\!\left(\gamma-\gamma'-\Delta\right), 
\end{equation}
and we prove the following theorem. 
\begin{theorem}\label{thm:short-intervals-keating}
Assume RH and let $0<\Delta = o(\log^2 T)$.
Then, as $T\to\infty,$
\begin{equation*}
\begin{split}
   \int\limits_0^T \left[S\!\left(t+\Delta\right)-S(t)\right]^2 \d t &=
    \frac{T}{\pi^2} \, \Bigg\{
    \frac{1}{2i}\int\limits_0^{\Delta}\left(\frac{\zeta'}{\zeta}(1+i t)-\frac{\zeta'}{\zeta}(1-i t)-\frac{2i\cos(t\log T)}{t}\,\right) \d t +  \widetilde{C}(\Delta)
    \\
    &\qquad\qquad\qquad + \frac{1}{2}\int\limits_{1}^{\infty} \frac{2F(\alpha) -F_{\Delta}(\alpha) -F_{-\Delta}(\alpha)}{\alpha^2}\, \d \alpha 
   \Bigg\}
    + o(T)
\end{split}
\end{equation*}
and
\begin{equation*}
\begin{split}
    \int\limits_0^T \left[\log \left|\zeta\!\left(\tfrac{1}{2}+i t+i\Delta\right)\right|-\log \left|\zeta\!\left(\tfrac{1}{2}+i t\right)\right|\right]^2 \d t 
    &
    =
    T \, \Bigg\{
    \frac{1}{2i}\int\limits_0^{\Delta}\left(\frac{\zeta'}{\zeta}(1+i t)-\frac{\zeta'}{\zeta}(1-i t)-\frac{2i\cos(t\log T)}{t}\,\right) \d t\\  
    &  \qquad\quad+ \widetilde{C}(\Delta)
     + \frac{1}{2} \int\limits_{1}^{\infty} \frac{2F(\alpha) -F_{\Delta}(\alpha) -F_{-\Delta}(\alpha)}{\alpha^2}\, \d \alpha 
   \Bigg\}
    + o(T),
\end{split}
\end{equation*}
where
\begin{equation}\label{eq:c-tilde-def}
    \widetilde{C}(\Delta) = 
    \sum_{\substack{m\ge 2}} \sum_p
    \left(\frac{1}{m^2}-\frac{1}{m}\right)\frac{1}{p^m}\Big(1-\cos\left(\Delta m\log p \right)\Big).
\end{equation}
\end{theorem}

\smallskip

We highlight that there is new input from the zeros contained in the function $F_\Delta$, and new input from the primes codified in the integral of the logarithmic derivatives of $\zeta(s)$ on the line $\re s = 1$. The integral involving $F_\Delta$ is convergent and remains bounded as $T\to\infty$ (see Lemma \ref{prop:BoundedOnAverage}). Conceivably, Theorem \ref{thm:short-intervals-keating} continues to hold in a much longer range of $\Delta$. 
We give two alternative formulations of this theorem. Our first reformulation better illustrates the connection to Fujii's previous result in \eqref{FujiiResult}.

\begin{theorem}\label{thm:short-intervals}
Assume RH and let $0<\Delta = o(\log^2 T)$. For $y\ge 1,$ define 
\begin{equation}\label{eq:E_def}
    E(y)=\sum_{n\le y}\Lambda^2(n)- y\log y+y.
\end{equation}
Then, as $T\to\infty,$
\begin{equation*}
\begin{split}
    \int\limits_0^T \left[S\!\left(t+\Delta\right)-S(t)\right]^2 \d t &=
    \frac{T}{\pi^2} \left\{
    \int\limits_{0}^{1} \frac{1-\cos (\Delta\alpha \log T )}{\alpha}\, \d \alpha
    + \frac{1}{2}\int\limits_{1}^{\infty} \frac{2F(\alpha) -F_{\Delta}(\alpha) -F_{-\Delta}(\alpha)}{\alpha^2}\, \d \alpha 
   \right\}\\
    &\qquad \qquad + \frac{T}{\pi^2}  \,c\!\left(\Delta\right)
    + o(T)
\end{split}
\end{equation*}
and
\begin{equation*}
\begin{split}
    \int\limits_0^T \Bigg[\log \left|\zeta\!\left(\tfrac{1}{2}+i t+i\Delta\right)\right|-\log& \left|\zeta\!\left(\tfrac{1}{2}+i t\right)\right| \Bigg] ^2 \d t \\
   &=T\left\{ \int\limits_{0}^{1} \frac{1-\cos (\Delta \alpha\log T )}{\alpha}\, \d \alpha
    + \frac{1}{2}\int\limits_{1}^{\infty} \frac{2F(\alpha) -F_{\Delta}(\alpha) -F_{-\Delta}(\alpha)}{\alpha^2}\, \d \alpha 
   \right\}\\
    &\qquad \qquad  + T\,c\!\left(\Delta\right)
    + o(T),
\end{split}
\end{equation*}
where
\begin{equation}\label{eq:c-function}
    c(v):=\int\limits_1^\infty \frac{E(y)}{y^2\log^3 y}\left[-v\log y\,  \sin\!\left(v\log y\right) + \sin^2\!\left(\frac{v\log y}{2}\right)\left(\log y + 2 \right)\right] \d y -\frac{v^2}{2}. 
\end{equation}
\end{theorem}

Using Theorem \ref{thm:short-intervals-keating}, the function $c(v)$ can also be written in terms of the Taylor series expansion of 
\[
\frac{\zeta'}{\zeta}(1+i t)-\frac{\zeta'}{\zeta}(1-i t)+\frac{2}{it}
\]
about $t=0$.
Note that, for $1\le y<2$, we have  $E(y)=-y\log y + y$, and the prime number theorem (unconditionally) implies that $E(y)=O_N(y/\log^N y)$, as $y\to \infty$, for any $N>0$. These two facts, together with the inequality $|\sin x|\le |x|$, imply that $c(v)$ is well-defined and that $c(v)\ll v^2$, for all $v\ge 0$. In particular, if $\Delta=o(1)$, as in Fujii's case in \eqref{FujiiResult}, then the term $T\,c\left(\Delta\right)=o(T)$ and can be absorbed into the error term. Moreover, when $\Delta=o(1)$, we show that this integral reduces to the analogous term in \eqref{FujiiResult} involving $F(\alpha)$ (see Section \ref{sec:transition}), recovering Fujii's result in this range. This reduction, while based on simple ideas, is quite subtle, and requires another technical but straightforward modification of Montgomery's theorem for $F(\alpha)$ to control some of the error terms. \smallskip

As explained above for Theorem \ref{thm:variance}, the proofs for the imaginary and real parts of $\log \zeta(\frac{1}{2}+it)$ are similar, but the proof for the real part is significantly more difficult. For this reason, we give the details only for the latter. Although we present the main steps of the proofs of Theorem \ref{thm:variance} and Theorem \ref{thm:short-intervals} in parallel, it is important to note that the proof of Theorem \ref{thm:variance} is independent of the proof of Theorem \ref{thm:short-intervals}. Additionally, we will use Theorem \ref{thm:variance} to control some of the error terms in some steps for Theorem \ref{thm:short-intervals} (see Lemma \ref{lem:H}  
below). \smallskip

Our second reformulation of Theorem \ref{thm:short-intervals-keating}  illustrates the input from the primes and the zeros in a simpler way. As we shall see in the next section, this has the advantage of allowing for a simple comparison with a conjecture of Berry \cite{Be}.
\begin{theorem}\label{thm:short-intervals-berry}
Assume RH and let $0<\Delta = o(\log^2 T)$.
Then, as $T\to\infty,$
\begin{equation*}
\begin{split}
   \int\limits_0^T \left[S\!\left(t+\Delta\right)-S(t)\right]^2 \d t &=
    \frac{T}{\pi^2} \, \Bigg\{
    \sum\limits_{n \leq T}\frac{\Lambda^2(n)}{n \log^2 n}\left(1-\cos\!\left(\Delta \log n \right)\right)\\
    &\qquad\qquad + \frac{1}{2}\int\limits_{1}^{\infty} \frac{2F(\alpha) -F_{\Delta}(\alpha) -F_{-\Delta}(\alpha)}{\alpha^2}\, \d \alpha 
   \Bigg\}
    + o(T)
\end{split}
\end{equation*}
and
\begin{equation*}
\begin{split}
    \int\limits_0^T \left[\log \left|\zeta\!\left(\tfrac{1}{2}+i t+i\Delta\right)\right|-\log \left|\zeta\!\left(\tfrac{1}{2}+i t\right)\right|\right]^2 \d t 
    &
    =
    T \, \Bigg\{
    \sum\limits_{n \leq T}\frac{\Lambda^2(n)}{n \log^2 n}\left(1-\cos\!\left(\Delta \log n \right)\right)\\
    &\qquad\quad + \frac{1}{2}\int\limits_{1}^{\infty} \frac{2F(\alpha) -F_{\Delta}(\alpha) -F_{-\Delta}(\alpha)}{\alpha^2}\, \d \alpha 
   \Bigg\}
    + o(T).
\end{split}
\end{equation*}
\end{theorem}

\indent The proofs of all of our theorems rely on knowledge of $F(\alpha)$ and $F_{\Delta}(\alpha)$ for $|\alpha| \le 1$. It is known that $F(\alpha)$ is real-valued, positive, and even. Moreover, refining Montgomery's original work \cite{Mo}, Goldston and Montgomery \cite{GoMo} showed that
\begin{equation}\label{eq:F_asymptotic}
    F(\alpha) = \left(T^{-2\alpha}\log T + \alpha
    \right)(1+o(1)),
\end{equation}
uniformly for $0\le \alpha \le 1$. Here, the term of $o(1)$ is of size $O\Big(\sqrt{\tfrac{\log \log T}{\log T}}\Big).$ In contrast, the function $F_\Delta(\alpha)$ is no longer positive, nor real, nor even; however, it satisfies the symmetry relations
\begin{equation}\label{eq:symmetry_rels}
    \overline{F_\Delta(\alpha)}=F_\Delta(-\alpha)=F_{-\Delta}(\alpha).
\end{equation}
Combining the methods of Chan \cite[Theorem 1.1]{Ch} with Goldston-Montgomery \cite{GoMo}, it can be shown that 
\begin{equation}\label{eq:F_delta_asymptotic}
    F_\Delta(\alpha) = T^{-2\alpha}\log T + \alpha\,  w\!\left(\Delta\right)T^{- i\alpha\Delta} + O\!\left(\frac{1}{\sqrt{\log T}}\right) + O(T^{-2\alpha})+ O_\varepsilon\!\left(\frac{(\Delta+1)\,T^{-\alpha\left(\tfrac{1}{2}-\varepsilon\right)}}{\log T}\right),
\end{equation}
uniformly for $0\le \alpha \le 1$ and small $\varepsilon>0$.

\subsection{A conjecture of Berry}\label{BS}
The Hilbert-P\'olya conjecture states that the imaginary parts of the zeros of $\zeta(s)$ correspond to the eigenvalues of some self-adjoint operator, and this would imply RH. In 1973, as a consequence of his work on the pair correlation of zeros, Montgomery \cite{Mo} was led to conjecture that the zeros of $\zeta(s)$ are distributed as the eigenvalues of a random matrix from the Gaussian unitary ensemble (GUE), giving support to a spectral interpretation of the zeta zeros. Montgomery's conjecture is supported by numerical evidence of Odlyzko \cite{Od}, which suggests that the GUE model holds for short-range statistics between zeros, such as the distribution of the gap between consecutive zeros. However, Odlyzko's evidence shows that the GUE model fails for long-range statistics, such as the correlation between zeros that are very far apart. In this case, Berry \cite{Be} suggested that these long-range statistics are better described in terms of primes, instead of GUE statistics. \smallskip

Berry \cite{Be} proposed a conjectural model for the zeros of $\zeta(s)$ as the eigenvalues of a quantum Hamiltonian operator. His model is expected to conform to the behavior of both short-range and long-range statistics of zeros, as described above. In 1988, Berry used his model to conjecture an asymptotic formula (in terms of our notation) for 
\begin{equation}\label{eq:S_short_intervals}
  \int\limits_0^T \!\left[S\!\left(t+\tfrac{2\pi\delta}{\log T}\right)-S(t)\right]^2 \d t.  
\end{equation}
As described above, the universal regime of his model is when $\delta=o(\log T)$, while the non-universal regime corresponds to $\delta \gg \log T$. 
\smallskip

We first briefly describe Berry's conjecture following his notation. For $E>0$, define
\begin{equation*}
    \mathcal{N}(E):=\frac{E}{2\pi}\left(\log \frac{E}{2\pi}-1\right)+\frac{7}{8},
\end{equation*}
so that, by \eqref{N(T)definition}, we have $N(E)=\mathcal{N}(E)+S(E)+O\!\left(\frac{1}{E}\right)$. For $m\ge1,$ let $x_m=\mathcal{N}(\gamma_m)$ be the renormalized zeros of $\zeta(s)$, so that the sequence $x_m$ has average spacing 1. In what follows, we let $E$, $x$, and $\Delta x$ be three large parameters, which we think of as going to infinity, satisfying the relations $\Delta x = o(x)$ and $x=\mathcal{N}(E)\sim \frac{E}{2\pi}\log E$.
For $L<x$, let $n(L;x)$ be the number of renormalized zeros $x_m$ in the interval $\left[x-\frac{L}{2}, x+\frac{L}{2}\right]$. In particular, note that 
\begin{equation*}
    n(L;x)=L + S\!\left(\mathcal{N}^{-1}\!\left(x+\frac{L}{2}\right)\right)
    - S\!\left(\mathcal{N}^{-1}\!\left(x-\frac{L}{2}\right)\right) + O\!\left(\frac{\log x}{x}\right).
\end{equation*}
We define the variance as
\begin{equation*}
    V(L;x) = \big\langle[n(L;x)-L]^2\big\rangle:=\frac{1}{\Delta x}\int\limits_{x-\frac{\Delta x}{2}}^{x+\frac{\Delta x}{2}} [n(L;y)-L]^2 \, \d y.
\end{equation*}
Finally, we let $\tau^*$ be another parameter, such that  $\tau^*=\frac{\Phi(E)}{\log (E/2\pi)}$ for some function $\Phi(E)$ that goes to infinity as $E\to\infty$. Berry's conjectural formula \cite[Eq.  (19)]{Be} states that \begin{align*}
    V(L;x)\sim &
    \frac{1}{\pi^2}\left[
    \log (2\pi L ) -\textup{Ci}(2\pi L) -2\pi L\Si(2\pi L ) +\pi^2 L -\cos(2\pi L) + 1 +\gamma_0
    \right] \\
    &+ \frac{1}{\pi^2}\left[
    2\sum_{r=1}^\infty\sum_p^{p^r<(E/2\pi)^{\tau^*}} \frac{\sin^2(\pi L r \log p / \log (E/2\pi))}{r^2 p^r} +\textup{Ci}(2\pi L\tau^*) - \log (2\pi L\tau^*)-\gamma_0
    \right],
\end{align*}
as $E\to\infty$, where $\gamma_0$ is Euler's constant,
\begin{equation}\label{eq: Sin_Cos_Int}
    \Si(x):=\int\limits_0^x \frac{\sin u}{u} \, \d u, \ \text{ and } \ \textup{Ci}(x):=-\int\limits_x^\infty  \frac{\cos u}{u} \, \d u.
\end{equation}
The right-hand side does not depend on the choice of $\tau^*$, as $E\to\infty$. The universal regime is when $L=o( 1/\tau^*)$, where the term in the first set of brackets is the dominant (leading order) term. See also \cite[Eqs. (20) and (21)]{Be} for simplifications in different ranges of $L$. Translating to our normalization and our notation, Berry made the following conjecture. 

\begin{conjecture}[Berry, 1988]\label{con:berry}
Let $\delta>0$. Then, as $T\to\infty$, the following asymptotic formulae hold.
\begin{enumerate}
    \item [\textup{(}a\textup{)}] If $\delta = o(\log T)$, then 
\begin{equation*}
    \int\limits_0^T \!\left[S\!\left(t+\tfrac{2\pi\delta}{\log T}\right)-S(t)\right]^2 \d t = \frac{T}{\pi^2} \left[\log (2\pi \delta ) -\textup{Ci}(2\pi\delta) -2\pi\delta\Si(2\pi\delta ) +\pi^2\delta -\cos(2\pi\delta) + 1 +\gamma_0\right] + o(T). 
\end{equation*}
    \item [\textup{(}b\textup{)}]  If $\delta \gg \log T$, then
\begin{equation*}
    \int\limits_0^T \left[S\!\left(t+\tfrac{2\pi\delta}{\log T}\right)-S(t)\right]^2 \d t = \frac{T}{\pi^2} \left[
    \sum\limits_{n \leq T}\frac{\Lambda^2(n)}{n \log^2 n}\left(1-\cos\!\left(\frac{2\pi\delta\,\log n}{\log T}\right)\right)
    +1\right] + o(T). 
\end{equation*}
\end{enumerate}
\end{conjecture}

\smallskip

In 1990, Fujii \cite{Fu1} proved an asymptotic formula for \eqref{eq:S_short_intervals}, assuming RH, in the universal regime where $\delta=o(\log T)$. 
In particular, assuming RH and Montgomery's Strong Pair Correlation Conjecture, Fujii proves Berry's conjecture in the universal regime (part (a) of the above conjecture). However, Fujii's proof relies on the fact that $\tfrac{\delta}{\log T} \rightarrow 0$ as $T \rightarrow \infty$ in numerous places, and it is not obvious that his proof can be modified to establish part (b) up to an error of size $o(T)$. 

\smallskip

Assuming RH and a version of the Strong Pair Correlation Conjecture (in longer ranges) due to Chan, we show that our formulae in Theorems \ref{thm:short-intervals-keating} -- \ref{thm:short-intervals-berry} imply Berry's conjecture in both the universal and the non-universal regimes. Although Berry never conjectures the range of $\delta$ for which part (b) of Conjecture \ref{con:berry} holds, we verify his conjecture holds in the range $\delta=o(\log^\frac{4}{3} T)$. Conceivably part (b) continues to hold for $\delta$ in a much longer range. We require the following generalization of the strong form of Montgomery's Pair Correlation Conjecture due to Chan \cite[Conjecture 1.1]{Ch}.
\begin{conjecture}[Chan, 2004]\label{con:chan} 
For $|\alpha|\ge 1$ and  $\Delta=o\!\left(\log^{\frac{1}{3}}T\right), $ we have \[F_\Delta(\alpha) = T^{- i \alpha \Delta}\,w\!\left(\Delta\right)\big(1+o(1)\big),\] uniformly for $\alpha$ in compact intervals as $T\to \infty$. 
\end{conjecture}

\begin{corollary}\label{cor:berry}
Assume RH and Conjecture \ref{con:chan}. Then, Conjecture \ref{con:berry} holds for all $\delta=o\!\left(\log^{\frac{4}{3}}T\right)$.
\end{corollary}
 The restriction on $\delta$ in the above corollary comes from Conjecture \ref{con:chan}. As we shall see, the restriction $\Delta=o(\log^2 T)$ in Theorems \ref{thm:short-intervals-keating} -- \ref{thm:short-intervals-berry}, corresponding to $\delta=o(\log^3 T)$, arises naturally in two different places. It first arises from the last error term in the formula \eqref{eq:F_delta_asymptotic} for $F_\Delta(\alpha)$ from \cite[Theorem 1.1]{Ch} (see Lemma \ref{prop:BoundedOnAverage} and Lemma \ref{lem:r}), and it again arises from estimating a sum over primes in Lemma \ref{lem:h+g} below. In the following sections, we attempt to state each lemma in the largest possible range of $\Delta$ to clarify where these restrictions appear. 

\section{A representation formula for $\log|\zeta(1/2+it)|$}\label{sec:representation}
\subsection{Some auxiliary functions}
Following the ideas developed by Goldston \cite{Go}, we must obtain a representation formula for $\log |\zeta(1/2+it)|$ in terms of a Dirichlet polynomial supported over prime powers and a sum over the zeros of $\zeta(s)$. This is based on an explicit formula of Montgomery \cite{Mo} and, in our case, requires introducing three auxiliary, real-valued functions, whose technical properties play important roles in our proof. For $u\in (0,2),$ define
	\begin{equation}\label{eq:def_f}
	    f(u):= u\int\limits_0^\infty \frac{\sinh[y(1-u)]}{\cosh y}\d y;
	\end{equation}
	for $u\in (-2,2),$ define
	\begin{equation}\label{eq:def_g}
g(u) :=  \int\limits_{0}^{\infty} \frac{e^{-y}\cosh(uy)}{\cosh y} \,\d y;
	\end{equation}
	and for $u\in \R\setminus\{0\},$ define
	\begin{equation}\label{eq:def_h}
	    h(u):= \cos u \int\limits_0^\infty \frac{y}{\cosh y}\frac{\d y}{y^2\!+\!u^2}.
	\end{equation}
	Before stating our representation formula, we collect relevant properties of $f$, $h$ and $g$ in the following lemma. This is similar to the ideas and functions used in \cite[Lemma 4 and Lemma 5]{CQ-H2}, and they will be used to obtain a delicate cancellation of main terms. 
	\begin{lemma}\label{lem:functions}
	Let $f$, $h$, and $g$ be defined in \eqref{eq:def_f}, \eqref{eq:def_h}, and \eqref{eq:def_g}, respectively. Then
\begin{enumerate}
    \item [\textup{(}a\textup{)}] we have $g\in C^\infty(-2,\,2)$ and $g$ is even\textup{;}
    \item [\textup{(}b\textup{)}] for $u \in (0,2)$, 
we have 
\begin{equation*}
g(u) = \frac{1-f(u)}{u} ;
\end{equation*}
\item [\textup{(}c\textup{)}] we have $h\in L^1(\R)\cap L^2(\R)$, $h$ is even, and
\begin{equation*}\label{hhatFunction}
\widehat{h}(a)=\pi \begin{cases} 
g(2\pi a), &  \textrm{if}\ \  2\pi |a|\leq 1, \\
 \dfrac{1}{2\pi |a|}, & \textrm{if}\ \   2\pi |a|>1.
   \end{cases}
\end{equation*} 
\end{enumerate}
	\end{lemma}
	{\sc Remark}. We highlight that $h$ has an unbounded but integrable singularity at the origin, which is different from the situation in both \cite{Go} and \cite{CQ-H2}. We also note that $f,$ $g,$ and their derivatives are uniformly bounded on the interval $[0,\,1]$.
	\begin{proof}
First we consider $g(u)$ as defined in \eqref{eq:def_g}. By the Dominated Convergence Theorem, we have that $g\in C^{\infty} (-2,2)$.  The fact that $g$ is even follows from the fact that $\cosh(y)$ is even. Now let $u \in (0,2)$. Then for all $u >0$, we know 
\begin{equation*}
\frac{1}{u} = \int\limits_0^{\infty} e^{-uy} \, \d\, y.
\end{equation*}
Using this representation for $\tfrac{1}{u}$, it follows that
\begin{equation*}
\frac{1-f(u)}{u}  = \int \limits_0^{\infty} \left( e^{-uy} - \frac{\sinh(y(1- u))}{\cosh y} \right) \d\, y = \int \limits_0^{\infty}  \frac{ e^{-y}\left(e^{uy}+ e^{-uy}\right)}{e^y+e^{-y}}  \d\, y = \int \limits_0^{\infty}  \frac{ e^{-y}\cosh(uy)}{\cosh y}  \d\, y=g(u),
\end{equation*}
as claimed. Next we consider $h(u)$ defined in \eqref{eq:def_h}.  First, note that $h(u)$ is even by construction. Next, we will show that $h\in L^1(\mathbb{R})$.  By the definition of $h$, observe that 
\begin{equation}\label{hFunctionbound}
\int\limits_{-\infty}^{\infty}\hspace{-.1cm} |h(v)| \d v  \leq   
2\int\limits_0^{\infty} \frac{1}{\cosh u}  \int\limits_{0}^{\infty}  \frac{u}{u^2+v^2}   \d v \d u= \pi\int\limits_0^{\infty} \frac{\d u}{\cosh u}= \pi^2\nonumber ,
\end{equation}
which implies $h\in L^1(\mathbb{R})$. Next, we calculate the Fourier transform of $h(v)$ using the well known Fourier pair 
\begin{equation}\label{firstfourierpair}
 \varphi(y)= e^{-2\pi |y|}\quad \text{ and } \quad \widehat{\varphi}(\xi) =\frac{1}{\pi} \frac{1}{1+\xi^2}.
\end{equation}
Let $a\in \mathbb{R}$. Since $h$ is even, we may assume $a\geq 0$. Using the variable change $w=\tfrac{v}{u}$ and \eqref{firstfourierpair}, it follows that

\begin{align*}\label{hhatFunctionPart1}
\widehat{h}(a) 
&=\int \limits_{-\infty}^{\infty} h(v)\, e^{-2\pi i a v}\, \d v \nonumber\\
&=\int \limits_{0}^{\infty}  \frac{u}{\cosh u} \int\limits_{-\infty}^{\infty} \frac{\cos v}{u^2 +v^2} \ e^{-2\pi i av}\, \d v\, \d u
 \nonumber\\
 &=\frac{1}{2} \int \limits_{0}^{\infty}  \frac{1}{\cosh u}  \int\limits_{-\infty}^{\infty} \frac{(e^{u(\frac{1}{2\pi}- a)2\pi i w}+e^{u(-\frac{1}{2\pi}-a)2\pi i w})}{1 +w^2} \d w\, \d u \nonumber\\
&=\frac{\pi}{2} \int \limits_{0}^{\infty}  \frac{1}{\cosh u} \big(e^{-u|1- 2\pi a|}+e^{-u|1+2\pi a|}\big)\, \d u \nonumber\\
      &=\begin{cases} 
 \pi \ g(2\pi a), & 0\le 2\pi a\leq 1, \\
 \dfrac{1}{2a}, & 2\pi a>1.
   \end{cases}\nonumber
\end{align*}
Clearly, $\widehat{h}\in L^2(\R)$, and therefore $h\in L^2(\R)$.  This completes the proof.
	\end{proof}
	
	\subsection{Representation formula}
\noindent We now state our formula for $\log|\zeta(\tfrac{1}{2}+it)|$.
	\begin{lemma}\label{lem:represLogZeta}
	Assume RH. For $x\ge 4$, $t\ge 1$, and $t\neq \gamma$, we have 
	\begin{equation*}
	    \log|\zeta(\tfrac{1}{2}+it)|= - \sum_\gamma h[(\gamma-t)\log x] 
	    +\sum_{n\le x} \frac{\Lambda(n)\cos(t\log n)}{n^{1/2}\log n}f\!\left(\frac{\log n}{\log x}\right)
	    + \frac{\log 2\, \log\frac{t}{2\pi} }{2\log x}
	    + O\!\left(\frac{x^{1/2}}{t\log^2 x} \right).
	\end{equation*}
	\end{lemma}
	\begin{proof}
We begin with the fact that, for $ t\neq \gamma$, we have 
\[
\log|\zeta(\tfrac{1}{2}+it)|  = -\int\limits_{1/2}^{\infty} \Re \frac{\zeta'}{\zeta} (\sigma+it)\d \sigma.
\] 
Now we use a slightly modified version of an explicit formula of Montgomery \cite{Mo} (see \cite[Eq. (2.6)]{CQ-H2}). For $\rho = \tfrac12 +i \gamma$, $x\geq4$, $s=\sigma+it$ with $\sigma>\tfrac{1}{2}$ and $t\geq 1$, 
we have
\begin{align}\label{eq:eqrealpart}
(x^{\sigma-\tfrac12}  + x^{\tfrac12 - \sigma}) \Re \frac{\zeta'}{\zeta} (\sigma+it)&=   \sum\limits_{\gamma} \cos((t-\gamma)\log x) \frac{2(\sigma-\tfrac12)}{(\sigma-\tfrac12)^2+(t-\gamma)^2}\nonumber   \\
& \quad-\sum\limits_{n \leq x} \Lambda(n)\cos(t\log n) \left(\frac{x^{\sigma - \tfrac12}}{n^{\sigma}} - \frac{x^{ \tfrac12-\sigma }}{n^{1-\sigma}} \right) \nonumber\\
& \quad -x^{\tfrac12}\Re \left( \frac{ x^{-it}(1-2\sigma)}{(\sigma-it)(1 -\sigma - it)}\right) -x^{\tfrac12-\sigma}\log\left( \frac{t}{2\pi}\right)\nonumber \\
&\quad  +O\!\left(\dfrac{\sigma^2}{t}\left(x^{-5/2}+x^{1/2-\sigma}\right)\right).
\end{align}
This immediately follows from \cite[Eq. (2.6)]{CQ-H2} by letting $n=1$ therein and taking real parts.
Dividing by $(x^{\sigma-\tfrac12}  + x^{\tfrac12 - \sigma}) = 2\cosh((\sigma-\tfrac12)\log x)$ and integrating \eqref{eq:eqrealpart} from $\tfrac12$ to infinity, for $x\geq 4,$ $t \geq 1$, and $t\neq \gamma$ we have
\begin{align}\label{eq:eqrealpart2}
\log|\zeta(\tfrac{1}{2}+it)| &=  
-  \sum\limits_{\gamma} \cos((t-\gamma)\log x) \int \limits_{1/2}^{\infty} \frac{(\sigma-\tfrac12)}{(\sigma-\tfrac12)^2+(t-\gamma)^2} \frac{\d \sigma}{\cosh((\sigma-\tfrac12)\log x)} \nonumber\\
& \quad+ \sum\limits_{n \leq x} \Lambda(n)\cos(t\log n)  \int \limits_{1/2}^{\infty} \left(\frac{x^{\sigma - \tfrac12}}{n^{\sigma}} - \frac{x^{ \tfrac12-\sigma }}{n^{1-\sigma}} \right) \frac{\d \sigma}{2\cosh((\sigma-\tfrac12)\log x)} \nonumber\\
& \quad + x^{1/2}\Re \left(x^{-it}  \int \limits_{1/2}^{\infty}  \frac{(\tfrac12-\sigma)}{(\sigma-it)(1 -\sigma - it)}\frac{\d \sigma}{\cosh((\sigma-\tfrac12)\log x)} \right) \nonumber\\
&\quad + \frac{1}{2}\log\frac{t}{2\pi} \int \limits_{1/2}^{\infty}  \frac{x^{\tfrac12-\sigma}}{\cosh((\sigma-\tfrac12)\log x)}\d \sigma 
+   O\!\left(\frac{1}{t} \int \limits_{1/2}^{\infty}  \frac{ \sigma^2 \left(x^{-5/2}+x^{1/2-\sigma}\right) }{\cosh((\sigma-\tfrac12)\log x)} \d \sigma \right).
\end{align}
By using the substitution, $u = (\sigma-\tfrac12)\log x$, 
the integral in the second main term of \eqref{eq:eqrealpart2} yields
\begin{equation*}
\begin{split}
\int \limits_{1/2}^{\infty} \left(\frac{x^{\sigma - \tfrac12}}{n^{\sigma}} - \frac{x^{ \tfrac12-\sigma }}{n^{1-\sigma}} \right) \frac{\d \sigma}{2\cosh((\sigma-\tfrac12)\log x)} 
= \frac{1}{n^{1/2}\log n}f\!\left( \frac{\log n}{\log x}\right),
\end{split}
\end{equation*}
where $f$ is defined in \eqref{eq:def_f}. 
Again, by the same substitution, for the first main term of \eqref{eq:eqrealpart2}, we have
\begin{equation*}
\begin{split}
-\sum\limits_{\gamma} \cos&((t-\gamma)\log x) \int \limits_{1/2}^{\infty} \frac{(\sigma-\tfrac12)}{(\sigma-\tfrac12)^2+(t-\gamma)^2} \frac{\d \sigma}{\cosh((\sigma-\tfrac12)\log x)} 
=-\sum\limits_{\gamma}h[(t-\gamma')\log x],
\end{split}
\end{equation*}
where $h$ is 
defined in \eqref{eq:def_h}.  Finally, the fourth term of \eqref{eq:eqrealpart2} equals
\begin{equation*}
\begin{split}
\frac{\log\tfrac{t}{2\pi}}{2} &\int \limits_{1/2}^{\infty}  \frac{x^{\tfrac12-\sigma}}{\cosh((\sigma-\tfrac12)\log x)}\d \sigma
=\frac{\log 2\log\tfrac{t}{2\pi}}{2\log x}.
\end{split}
\end{equation*}
The other terms are error terms and can be treated similarly to the proof of \cite[Lemma 1]{Go}. Combining all the terms of \eqref{eq:eqrealpart2} completes the proof.  
\end{proof}
Note that we have the extra main term $\frac{\log2 \log t/2\pi}{\log x}$ when compared to Goldston's formula for $S(t)$ in \cite[Lemma 1]{Go}. This comes from Stirling's formula when analyzing the real part of $\frac{\zeta'}{\zeta}(s)$, and it does not appear when taking the imaginary part. We use Lemma \ref{lem:represLogZeta} to obtain an expression for the quantities we want to compute in Theorems \ref{thm:variance} and \ref{thm:short-intervals}. We now adopt some notation for the expressions we will consider. Henceforth, let $T\ge 4$ and $\Delta=\Delta(T)$ be a function of $T$ such that $0<\Delta\ll T^b$, for some fixed $0<b<1$. For $t\ge 1$, denote 
	\begin{equation}\label{eq:A_B_defs}
	    A(t):= \sum_{n\le x} \frac{\Lambda(n)\cos(t\log n)}{n^{1/2}\log n}f\!\left(\frac{\log n}{\log x}\right) \ \ \ \ \text{and}\ \ \ \ B(t):=  - \sum_\gamma h[(\gamma-t)\log x], 
	\end{equation}
	so that $A(t)$ contains the information on primes and $B(t)$ contains the information on zeros in our expression for $\log |\zeta(1/2+it)|$. Additionally, denote
	\begin{align}\label{eq:GHR_defs}
	    &G_1:= -\int\limits_{1}^T |A(t)|^2\,\d t, \ \ \ \ \ \ \ \ \ \ 
	    G_2:= -\int\limits_1^T \left|A\left(t+\Delta\right) - A(t)\right|^2\,\d t,\nonumber\\ &H_1:=2\int\limits_{1}^T A(t)\,\log |\zeta(\tfrac{1}{2}+it)| \,\d t, \nonumber\\
	    &H_2:= 2\int\limits_1^T \left[A\!\left(t+\Delta\right)-A(t)\right]
	    \left[\log \left|\zeta\!\left(\tfrac{1}{2}+it+i\Delta\right)\right|-\log \left|\zeta(\tfrac{1}{2}+it)\right|\right]\, \d t,\nonumber\\
	    &R_1:= \int\limits_1^T |B(t)|^2\, \d t, \ \ \ \ \ \ \ \ \ \ 
	    R_2:= \int\limits_1^T\left|B\!\left(t+\Delta\right)-B(t)\right|^2\, \d t.
	\end{align}
	In the next result, we use Lemma \ref{lem:represLogZeta} to write the objects in Theorems \ref{thm:variance} and \ref{thm:short-intervals} in terms of the above expressions $G_i,\, H_i,\, $ and $R_i.$
	\begin{lemma}\label{lem:represFinal} Assume RH. 
	Let $4\le x\le T$ and let $0 < \Delta \ll T^b$, where $b<\frac{1}{2}$. Then, as $T\to\infty$, we have
	\begin{align*}
	    &\textup{(a)} \ \int\limits_1^T \log^2 |\zeta(\tfrac{1}{2}+it)| \,\d t = G_1+H_1+R_1 -T\log^2 T\frac{\log^2 2}{4\log^2 x} 
	    +O\!\left(\frac{T\log T}{\log^2 x}
	    \right) + O\!\left(\frac{\sqrt{xR_1} }{\log^2 x}\right);\\
	    &\textup{(b)} \  
	    \int\limits_1^T \left[\log \left|\zeta\!\left(\tfrac{1}{2}+i t+i\Delta\right)\right|-\log \left|\zeta(\tfrac{1}{2}+i t)\right|\right]^2 \d t 
	    = 
	    G_2 + H_2 + R_2 +  O\!\left(\frac{T}{\log^4 x}\right) + O\!\left(\frac{\sqrt{TR_2}}{\log^2 x}\right)
	    .
	\end{align*}
	\end{lemma}
	\begin{proof}
By rearranging the terms in Lemma \ref{lem:represLogZeta}, we have
\begin{equation*}
\begin{split}
 B(t) + O\!\left(\frac{x^{1/2}}{t\log^2 x}\right) =\log|\zeta(\tfrac{1}{2}+it)|  &-  A(t) -  \frac{\log 2\log\frac{t}{2\pi}}{2\log x}. 
\end{split}
\end{equation*}
Squaring the above expression and then integrating from $1$ to $T$ yields
\begin{align}\label{prop2aStep3}
R_1+ O&\!\left(\frac{\sqrt{xR_1}}{\log^2 x}\right) + O\!\left( \frac{x}{ \log^4 x}\right)\nonumber\\
&=\int\limits_1^T \log^2 |\zeta(\tfrac{1}{2}+it)| \, \d t -   H_1 - G_1 + 
\frac{\log^2 2}{4} \frac{T\log^2 T}{\log^2 x}+O\!\left( \frac{T\log T}{\log^2 x}\right)  
\nonumber\\
&\quad +O\!\left( \frac{1}{\log x}\left|\int\limits_1^TA(t) \log t \, \d t \right|\right) + O\!\left(\frac{1}{\log x} \left|\int\limits_1^T \log |\zeta(\tfrac{1}{2}+it)|\log t \, \d t\right|\right),
\end{align}
where we used the Cauchy-Schwarz inequality to bound the first error term on the left-hand side.
For the second error term on the right-hand side of \eqref{prop2aStep3}, since $f(v)$ is uniformly bounded for all $v \in [0,1]$, $|\cos(v)|\leq 1$ for all $v\in \mathbb{R}$, and  $ \int\limits_1^T n^{it} \log t \, \d t\ll\log T$ for $n\ge 2$, 
we see that
\begin{equation*}\label{prop2aStep7}
\begin{split}
\int\limits_1^TA(t)\frac{\log t}{\log x} \, \d t  &\ll  \frac{\log T}{\log x} \sum\limits_{n \leq x} \frac{\Lambda(n)}{n^{1/2}\log n} \ll \frac{\sqrt{x}}{\log x} \log T.
\end{split}
\end{equation*}
To control the last error term on the right-hand side of \eqref{prop2aStep3}, consider the antiderivative of $\log |\zeta(\tfrac{1}{2}+it)|$. Assuming RH, it is known that 
\[ 
\int\limits_1^T \log |\zeta(\tfrac{1}{2}+it)| \, \d t\ll \log T.
\]
(See \cite[Lemma 2.2]{BLM} for a slightly stronger estimate.) Thus, using integration by parts, we obtain
\begin{equation*}\label{prop2aStep8}
\begin{split}
\frac{1}{\log x} \int\limits_1^T \log |\zeta(\tfrac{1}{2}+it)|\log t \, \d t
\ll \frac{\log^2 T}{\log x}.
\end{split}
\end{equation*}
By combining and rearranging all the calculations for the terms in \eqref{prop2aStep3}, 
we complete the proof of part (a). 

For the proof of part (b), since $\Delta \ll T^b$,
we observe that for $t>1$, $x\geq 4$, and $\varepsilon>0$ the mean-value theorem implies that
\begin{align}\label{prop2bStep4}
\frac{\log 2\left[\log\!\left(\tfrac{t+\Delta}{2\pi}\right) -  \log \tfrac{t}{2\pi}\right]}{2\log x} 
\ll_\varepsilon  \frac{T^{\frac{1}{2}-\varepsilon}}{t\log x}, 
\nonumber
\end{align}
 so that this term is absorbed into the error bound. The rest of the proof of part (b) is analogous to the proof of part (a). Consequently, the proof is complete.
\end{proof}

In order to conclude the proofs of Theorems \ref{thm:variance} and \ref{thm:short-intervals}, the following sections are devoted to estimating the quantities $G_i$, $H_i$ and $R_i$.
	
\section{Contributions from the zeros}\label{sec:zeros}
\subsection{Auxilliary lemmas}
Before we compute $R_i$, we remark that the constant $a$ as defined in \eqref{GoldstonConstant} actually has a mild dependence on $T$, since it is defined in terms of $F(\alpha, T)$. In this subsection, we collect several useful technical estimates regarding the zeros and the function $F_\Delta(\alpha, T)$, and we show that this dependence on $T$ can be controlled in the proofs of our main theorems.

\begin{lemma}\label{prop:BoundedOnAverage}
Assume RH. Let $T\ge 4$ and $\Delta=O(\log^{2} T)$. Then, for $\beta> 0$, we have
\begin{equation*}
    \int\limits_0^\beta \Big( 2F(\alpha)-F_\Delta(\alpha)-F_{-\Delta}(\alpha) \Big) \, \d \alpha \ll (1+\beta)\left(1+\frac{|\Delta|}{\log^2 T} \right)
\end{equation*}
 and   
\begin{equation*}
  \int\limits_1^\infty \frac{2F(\alpha)-F_\Delta(\alpha)-F_{-\Delta}(\alpha)}{\alpha^2}\, \d \alpha \ll 1+\frac{|\Delta|}{\log^2 T},
\end{equation*}
where the implied constants are universal.
\end{lemma}
\begin{proof}
Consider the identity
\begin{equation}\label{eq:F_delta-id-positive}
    2F(\alpha) -F_{\Delta}(\alpha)-F_{-\Delta}(\alpha) = 
\frac{8\pi^2}{T\log T} \int\limits_{-\infty}^\infty e^{-4\pi |u|}\left[1-\cos\!\left(\Delta \alpha\log T+2\pi\Delta u\right)\right] \left|\sum_{0<\gamma\le T} T^{i\alpha \gamma}e^{2\pi i u \gamma}
    \right|^2 \d u.
\end{equation}
In particular, $2F(\alpha)-F_\Delta(\alpha)-F_{-\Delta}(\alpha)\ge 0$. Lemma \ref{prop:BoundedOnAverage} follows by modifying an argument of Goldston \cite[Lemma A]{Go} in a straightforward manner and applying Chan's theorem for $F_\Delta(\alpha)$ in the form given in \eqref{eq:F_delta_asymptotic}.
\end{proof}
\begin{lemma}\label{lem:ZeroEstimate}
 Let $T\ge 4$, $0\le |\Delta| \le T$, $0<H\le T$, and $w(u)=\frac{4}{4+u^2}$. Then, 
 \begin{align*}
	    &\textup{(a)} \sum_{0<\gamma, \gamma'\le T}|\gamma-\gamma'-\Delta|  \, w(\gamma-\gamma'-\Delta) \ll T\log^3 T;\\
	    &\textup{(b)} \sum_{T-|\Delta|-1\le \gamma, \gamma'\le T+H} w(\gamma-\gamma'-\Delta)\ll (H+|\Delta|+1)\log^2 T;\\
	    &\textup{(c)}\sum_{\substack{0< \gamma<T-|\Delta|-1
	    \\ T\le \gamma'\le T+H}} w(\gamma-\gamma'-\Delta)\ll \log^3 T;\\
	\end{align*}
 where the implied constants are universal.
\end{lemma}
\begin{proof}
 By interchanging $\gamma$ and $\gamma'$, we may assume that $\Delta\ge0$. We have
 \begin{align*}
  \sum_{0<\gamma, \gamma'\le T}|\gamma-\gamma'-\Delta| \, w(\gamma-\gamma'-\Delta) =& \sum_{\substack{0<\gamma, \gamma'\le T \\ \gamma-\gamma'-\Delta<0}}(\Delta+\gamma'-\gamma) \, w(\gamma-\gamma'-\Delta) \\
  +& \sum_{\substack{0<\gamma, \gamma'\le T \\ \gamma-\gamma'-\Delta>0}}(\gamma-\gamma'-\Delta) \, w(\gamma-\gamma'-\Delta) \\=& Z_1+Z_2,
 \end{align*}
 say. We use the inequality 
 \begin{equation*}
  \frac{4|u|}{4+u^2}\le \min\left(1, \frac{4}{|u|}\right)
 \end{equation*}
  and the fact that there are $O(\log T)$ zeros in the interval $[T-\Delta-2,T-\Delta]$ to estimate $Z_1$ as follows:
 \begin{align*}
  Z_1 &\le  \sum_{0<\gamma\le T}\ \sum_{\gamma-\Delta<\gamma'<\gamma-\Delta+2}1 
  +  \sum_{0<\gamma\le T}\ \sum_{\gamma-\Delta+2<\gamma'\le T} \frac{4}{\Delta+\gamma'-\gamma}\\
  &\ll \sum_{0<\gamma\le T}\log T + \sum_{0<\gamma\le T}\,\sum_{\gamma+2<n\le T+\Delta}\frac{\log n}{n-\gamma} \\
  &\ll T\log^2T + \sum_{0<\gamma\le T}\log^2(\Delta +T)\\
  &\ll T\log^3 T,  
 \end{align*}
 since $\Delta\le T$. The bound $Z_2\ll T\log^3 T$ is similar. This proves part (a).
 
\smallskip 
 
For part (b), since $0<H\le T$, we use that for $0\le n \le T+H$) there are $O(\log T)$ zeros in the interval $(n,n+1)$ to obtain:

\begin{align*}
 \sum_{T-|\Delta|-1\le \gamma, \gamma'\le T+H} w(\gamma-\gamma'-\Delta) &\ll 
 \log T \sum_{T-|\Delta|-1\le \gamma\le T+H} \ \sum_{0\le n\le H+|\Delta|+1} \frac{1}{1+(\gamma -T+|\Delta|+1-n-\Delta)^2} \\
 &\ll \log T \sum_{T-|\Delta|-1\le \gamma\le T+H} 1 \ll (H+|\Delta|+1)\log^2 T.
\end{align*}
In the last line, since the summand is positive, we may bound the sum over $n$ by a sum over all integers and then use the fact that  function 
\[
 \sum_{n\in\Z} \frac{1}{1+(x+n)^2}
\]
converges to a continuous periodic function of $x\in\R$. In particular, it is uniformly bounded.

\smallskip

For part (c), note that, for $0<\gamma<T-|\Delta|-1$ and $T\le \gamma'\le T+H$, we have $|\gamma -\gamma'-\Delta|\ge T+\Delta-\gamma\ge 1.$ Then, using that $w(u)\le\frac{4}{u^2}$, we have
\begin{align*}
 \sum_{\substack{0< \gamma<T-|\Delta|-1
	    \\ T\le \gamma'\le T+H}} w(\gamma-\gamma'-\Delta) &\ll 
	    \log T\sum_{0<\gamma<T-|\Delta|-1}\ \sum_{0\le n\le H+1} \frac{1}{(T+n+\Delta-\gamma)^2}\\
	    &\ll \log T\sum_{0<\gamma<T-|\Delta|-1}\frac{1}{T+\Delta-\gamma}\\ 
	    &\ll \log T \sum_{1\le n\le T+|\Delta|}\frac{\log n}{n} \ll \log^3 T,
\end{align*}
since $|\Delta|\le T$. This completes the proof of the lemma.
\end{proof}

\begin{lemma}\label{lem:a(T)} Assume RH. For $T\ge 4$, let $|\Delta|\le \log^{2} T$, and let $0<H\le T$. We have
\begin{align*}
     &(\textup{a})\  \int_1^\infty \frac{F(\alpha,T\!+\!H)}{\alpha^2}\, \d \alpha = \int_1^\infty \frac{F(\alpha,T)}{\alpha^2}\, \d \alpha + O\Big( \frac{(H+1)}{T} \log^3T \Big); \\
     &(\textup{b})\  \int_1^\infty \frac{2F(\alpha,T\!+\!H)-F_\Delta(\alpha,T\!+\!H)-F_{-\Delta}(\alpha,T\!+\!H)}{\alpha^2}\, \d \alpha = \int_1^\infty \frac{2F(\alpha,T)-F_\Delta(\alpha,T)-F_{-\Delta}(\alpha,T)}{\alpha^2}\, \d \alpha \\
     &\hspace{9.4cm}+ O\Big( \frac{(H+|\Delta|+1)}{T} \log^3T \Big). 
\end{align*}
Here, the implied constants are universal.
\end{lemma}
\begin{proof}
First, we prove a pointwise estimate for $F_\Delta(\alpha, T+H)$ that holds in the larger range $|\Delta|\le T$ and is useful for both parts (a) and (b). By the mean-value theorem, for $\theta\in \R$, we have
 \begin{equation*}\label{eq:mvt-F}
  |(T+H)^{i\theta}-T^{i\theta}|\ll \frac{H|\theta|}{T} \ \ \ \ \text{and} \ \ \ \ \frac{1}{(T+H)\log(T+H)}=\frac{1}{T\log T}\left(1+O\left(\frac{H}{T}\right)\right).
 \end{equation*}
Therefore, for $\Delta\in\R$ with $|\Delta|\le T$, we have
\begin{equation*}
 F_\Delta(\alpha, T+H) = \frac{2\pi}{T\log T}\sum_{0<\gamma,\gamma'\le T+H}(T+H)^{i\alpha(\gamma-\gamma'-\Delta)}w(\gamma-\gamma'-\Delta) + O\left(\frac{H}{T}F_\Delta(0,T+H)\right).
\end{equation*}
To bound the last error term, one can see that $|F_\Delta(\alpha,T)|\le F_\Delta(0,T) \ll \log T$ (uniformly for $0\le |\Delta|\le T$), analogously to the classical bound for $\Delta=0$. Now, to estimate the difference $|F_\Delta(\alpha, T+H)-F_\Delta(\alpha,T)|$, we separate the double sum over zeros in $F_\Delta(\alpha, T+H)$ depending on whether the zeros lie in the interval $(0,T]$ or $(T, T+H]$. Using the triangle inequality, we obtain
\begin{align*}
 |F_\Delta(\alpha,T+H)-F_\Delta(\alpha,T)| &\le \frac{2\pi}{T\log T}\sum_{0<\gamma,\gamma'\le T}\big|(T+H)^{i\alpha(\gamma-\gamma'-\Delta)}-T^{i\alpha(\gamma-\gamma'-\Delta)}\big|w(\gamma-\gamma'-\Delta) \\
 &\quad+  
 \frac{2\pi}{T\log T}\sum_{T<\gamma,\gamma'\le T+H}w(\gamma-\gamma'-\Delta)\\
 &\quad+
 \frac{2\pi}{T\log T}\sum_{\substack{0<\gamma\le T\\
 T< \gamma'\le T+H} 
 }
 w(\gamma-\gamma'-\Delta)\\
 &\quad+
 \frac{2\pi}{T\log T}\sum_{\substack{0<\gamma'\le T\\
 T< \gamma\le T+H }
 }w(\gamma-\gamma'-\Delta)+O\left(\frac{H}{T}\log T\right)\\
 &= Y_1+Y_2+Y_3+Y_4+O\left(\frac{H}{T}\log T\right),
\end{align*}
say. By the mean-value theorem and part (a) of Lemma \ref{lem:ZeroEstimate}, we find that $Y_1\ll \frac{H}{T}|\alpha|\log^2 T$. Since $w(u)\ge 0$, we may extend the sum in $Y_2$ to apply part (b) of Lemma \ref{lem:ZeroEstimate}. Therefore, $Y_2\ll \frac{1}{T}(H+|\Delta|+1)\log T.$ We estimate $Y_3$ by further dividing the sum into two parts:
\begin{align*}
 \sum_{\substack{0<\gamma\le T\\
 T< \gamma'\le T+H} 
 }
 w(\gamma-\gamma'-\Delta) &= \sum_{\substack{ T-|\Delta|-1\le\gamma\le T\\
 T< \gamma'\le T+H} 
 }
 w(\gamma-\gamma'-\Delta)
 +\sum_{\substack{0<\gamma< T-|\Delta|-1\\
 T< \gamma'\le T+H} 
 }
 w(\gamma-\gamma'-\Delta) \\
 &\ll (H+|\Delta|+1)\log^2 T + \log^3 T, 
\end{align*}
where we used that $w(u)\ge 0$ to extend the first sum and applied parts (b) and (c) of Lemma \ref{lem:ZeroEstimate}, respectively. This yields $Y_3\ll \frac{1}{T}(H+|\Delta|+1)\log^2 T$. $Y_4$ can be treated similarly to $Y_3$, since we may interchange $\gamma$ and $\gamma'$, use that $w(u)$ is even, and replace $\Delta$ with $-\Delta$. Combining the above estimates, we obtain that
\begin{equation}\label{eq:F_T+H}
 F_\Delta(\alpha,T+H)=F_\Delta(\alpha,T)+O\left(\frac{(|\alpha|+1)(H+|\Delta|+1)\log^2 T}{T}\right) ,
\end{equation}
uniformly for $\alpha\in\R$, $T\ge 4$, $0<H\le T$, and $\Delta\in\R$ with $0\le |\Delta|\le T$.\smallskip 

We now use the pointwise estimate \eqref{eq:F_T+H} to prove part (a) as follows. It is known that 
\begin{equation}\label{eq:F_bounded}
 \int_0^\beta F(\alpha,T)\,\d\alpha \ll 1+\beta, 
\end{equation}
uniformly for $\beta\ge0$ and $T\ge 4$ (see \cite[Lemma A]{Go}). Integrating by parts, for $\beta\ge 1$ this implies that 
\begin{equation*}
 \int_\beta^\infty \frac{F(\alpha, T)}{\alpha^2}\,\d\alpha \ll \frac{1}{\beta}.
\end{equation*}
Therefore, by the case $\Delta=0$ of the estimate \eqref{eq:F_T+H}, we obtain
\begin{align*}
 \int_1^\infty \frac{F(\alpha,T+H)}{\alpha^2}\,\d\alpha &=  \int_1^T \frac{F(\alpha,T+H)}{\alpha^2}\,\d\alpha + O\!\left(\frac{1}{T}\right) \\
 &= \int_1^T \frac{F(\alpha,T)}{\alpha^2}\,\d\alpha + O\!\left(\frac{(H+1)\log^3T}{T}\right)\\
 &=\int_1^\infty \frac{F(\alpha,T)}{\alpha^2}\,\d\alpha + O\!\left(\frac{(H+1)\log^3T}{T}\right). 
\end{align*}
This proves part (a). Part (b) is similar, using that $|\Delta|\le\log^{2} T$ and Lemma \ref{prop:BoundedOnAverage} in place of \eqref{eq:F_bounded}.
\end{proof}
\subsection{Unbounded discontinuities}
In this section, our goal is to express $R_i$ as a sum over pairs of zeros of $\zeta(s)$ in order to apply Montgomery's pair correlation method to estimate $R_i$. The arguments of Montgomery and Goldston consist of localizing the sum to zeros in the interval $[0,\, T]$ and then extending the integral in the definition of $R_i$ in \eqref{eq:GHR_defs} to infinity, up to small errors. However, due to the unbounded discontinuity of our weight function $h$ at the origin, their arguments do not apply directly. This leads to difficulties, and we must use a different and delicate approach to control the error terms in this case. The first part of this approach lies in the introduction of a sequence of $T_n$'s for which the following lemmas will hold. The idea of using such a sequence is classical (for instance, see \cite[Ch.17]{DP}). Since $N(T+1) - N(T)\ll \log T$, by the pigeonhole principle, for every $n\in\mathbb{N}$ we can find a sequence $\{T_n\}$ satisfying 
\begin{equation}\label{TnSequence}
n \le T < n+1 \ \text{ and } \ | \gamma - T_n|\gg \frac{1}{\log n}.
\end{equation}
In this way, we obtain similar results to Goldston on a sequence of points tending to infinity, despite the unbounded discontinuity of our function $h$. Now, we define
\begin{equation}\label{eq:k_def}
    k(\xi):=\frac{1}{\pi^2} \, \widehat{h}(\xi)^2,
\end{equation}
and we prove the following lemma.
\begin{lemma}\label{lem:unbounded_discontinuities}
Assume RH. Let $T\in \{T_n\}$, where $T_n$ satisfies \eqref{TnSequence}. Define $k$ as in \eqref{eq:k_def} and $R_i$ as in \eqref{eq:GHR_defs}. For $4\le x\le T$ and $0<\Delta\ll T^b$, with $0<b<\tfrac{1}{2}$, we have 
\begin{align*}
     &(\textup{a})\  R_1=\frac{\pi^2}{\log x}\sum_{0<\gamma,\, \gamma'\le T}\widehat{k}[(\gamma-\gamma')\log x] + O\!\left(\sqrt{T}\log^2 T\right);\\
     &(\textup{b})\  R_2 = \frac{2\pi^2}{\log x}\sum_{0 < \gamma, \gamma' \le T} \left\{\widehat{k}[(\gamma - \gamma' )\log x]-  \widehat{k}\left[\left(\gamma - \gamma' -\Delta \right)\log x\right] \right\}+  O\!\left(T\sqrt{\frac{\log \log T}{\log T}}\right).
\end{align*}
\end{lemma}
\begin{proof}[Proof of part (a)]
First note that for $\gamma \neq t$, using an argument of Goldston \cite[p.~158]{Go}, we find that
\begin{equation}\label{eq:SumOverZerosEst1}
\sum\limits_{\gamma} h[(t-\gamma)\log x] \ = \sum\limits_{|t-\gamma|\le \tfrac{1}{\log x}} h[(t-\gamma)\log x] +O(\log \tau ),
\end{equation}
since $h(v) \ll \tfrac{1}{v^2}$ for $|v|>1$. 
Here, $\tau=|t|+2.$ Similarly, modifying an argument of Montgomery \cite[p.~187]{Mo}, we deduce that for $t \in [0,T]$ we have
\begin{equation}\label{eq:SumOverZerosEst2}
\sum\limits_{\substack{\gamma \\ \gamma \notin [0,T]}} h[(t-\gamma)\log x] =  \sum\limits_{\gamma \in I }  h[(t-\gamma)\log x]  + O\Big( \Big[\tfrac{1}{T-t+1} + \tfrac{1}{t+1} \Big] \log T \Big),
\end{equation}
where $I =\{\gamma  :  T< \gamma \leq T  + \tfrac{1}{\log x} \}$. We now show that the terms in the sum for which $\gamma\notin [0,T]$ contribute an amount of size $o(T)$ to $R_1$. Using  \eqref{eq:SumOverZerosEst1} and \eqref{eq:SumOverZerosEst2}, we restrict the interval of zeros within the sum in $R_1$ to $\gamma,\gamma' \in [0,T]$.  Then by expanding the integral, we rewrite $R_1$ as \vspace{-.2cm}
 \begin{align}\label{R1Rep1Step1}
R_1 
 &=\!\!\!\!\!\!\!\sum \limits_{0< \gamma,\gamma'\leq T} \int\limits_{1}^{T} h[(t-\gamma)\log x] \, h[(t-\gamma')\log x]\, \d t +
 O\bigg(\int\limits_1^T \sum\limits_{\gamma \in I} \big |h[(t-\gamma)\log x]\big | \hspace{-.4cm} \sum\limits_{|t-\gamma'|\le \tfrac{1}{\log x}}\hspace{-.5cm}  \big |h[(t-\gamma')\log x]\big | \, \d t \bigg)  \nonumber\\
&  \hspace{-.1cm} \quad +O\bigg(\int\limits_1^T \hspace{-.1cm} \sum\limits_{\gamma \in I } \big |h[(t-\gamma)\log x]\big | \log \tau\, \d t \hspace{-.1cm} \bigg) + O\bigg( \hspace{-.1cm} \log T \hspace{-.15cm}  \int \limits_1^T\hspace{-.1cm} \Big[\tfrac{1}{T-t+1} + \tfrac{1}{t+1} \Big]  \log \tau \, \d t \hspace{-.1cm}  \bigg) \nonumber\\
 &  \hspace{-.1cm}  \quad+O\bigg(\hspace{-.1cm} \log T \int\limits_1^T\hspace{-.1cm}\Big[\tfrac{1}{T-t+1} + \tfrac{1}{t+1} \Big] \hspace{-.4cm} \sum\limits_{|t-\gamma'|\le \tfrac{1}{\log x}} \hspace{-.5cm}\big |h[(t-\gamma')\log x]\big |  \, \d t \hspace{-.1cm} \bigg).
\end{align}
Integrating the third error term on the right-hand side of \eqref{R1Rep1Step1} gives
\begin{align}\label{R1Rep1Step2}
\log T \hspace{-.15cm}  \int \limits_1^T\hspace{-.1cm} \Big[\tfrac{1}{T-t+1} + \tfrac{1}{t+1} \Big]\, \hspace{-.1cm} \log t \, \d t 
\ll \log^3 T.\nonumber
\end{align}
Using the facts that $h\in L^1$ and 
\[
|I| \ll \frac{\log T}{\log x} + \frac{\log T}{\log\log T},
\]
the second error term on the right-hand side of \eqref{R1Rep1Step1} reduces to 
\begin{equation*}\label{R1Rep1Step3}
\int\limits_1^T \sum\limits_{\gamma \in I } \big |h[(t-\gamma)\log x]\big | \log t\, \d t \ll \log T \sum\limits_{\gamma \in I }   \int\limits_{-\infty}^{\infty}  \big |h[(t-\gamma)\log x]\big |\, \d t \ll \frac{\log^2 T}{\log x}.
\end{equation*}
Similarly, the fourth error term on the right-hand side of \eqref{R1Rep1Step1} yields 
\begin{equation*}\label{R1Rep1Step4}
\begin{split}
 \log T \int\limits_1^T\Big[\tfrac{1}{T-t+1} &+ \tfrac{1}{t+1} \Big]\hspace{-.5cm} \sum\limits_{|t-\gamma'|\le \tfrac{1}{\log x}} \hspace{-.5cm} \big |h[(t-\gamma')\log x]\big |\, \d t\\
&=   \log T \hspace{-.1cm}\int\limits_1^T \tfrac{1}{T-t+1} \hspace{-.5cm} \sum\limits_{|t-\gamma'|\le \tfrac{1}{\log x}} \hspace{-.5cm} \big |h[(t-\gamma')\log x]\big |  \, \d t +   \log T\hspace{-.1cm} \int\limits_1^T \tfrac{1}{t+1} \hspace{-.5cm} \sum\limits_{|t-\gamma'|\le \tfrac{1}{\log x}} \hspace{-.5cm} \big |h[(t-\gamma')\log x]\big |  \, \d t\nonumber\\
&= S_1+S_2.
\end{split}
\end{equation*}
We introduce a parameter $1<H<T$ to split the range of integration for $S_1$ as follows:
\begin{equation}\label{R1Rep1Step6}
\begin{split}
S_1&=   \log T \hspace{-.2cm} \int\limits_1^{T-H}  \hspace{-.2cm} \frac{1}{T-t+1}  \hspace{-.5cm}
\sum\limits_{|t-\gamma'|\le \tfrac{1}{\log x}}
\hspace{-.5cm} \big |h[(t-\gamma')\log x]\big |  \, \d t +
\log T \int\limits_{T-H}^T  \hspace{-.2cm}\frac{1}{T-t+1}  \hspace{-.5cm} \sum\limits_{|t-\gamma'|\le \tfrac{1}{\log x}}  \hspace{-.5cm}\big |h[(t-\gamma')\log x]\big |  \, \d t\nonumber\\
&= S_{11}+S_{12}.
\end{split}
\end{equation}
To estimate $S_{11}$, we note that $T-t+1\ge H+1,$ extend the sum over $\gamma'$, and use that $h\in L^1.$ We find that
\begin{equation*}
\begin{split}
    S_{11} &\ll \frac{\log T}{H+1} \sum\limits_{0\leq \gamma'\le T} \int\limits_1^{T-H}  \big |h[(t-\gamma')\log x]\big |  \, \d t \\
    &\ll \frac{T\log^2 T}{H+1}.
\end{split}
\end{equation*}
For $S_{12}$, we use that $T-t+1\ge1$ and extend the sum slightly to obtain
\begin{equation*}
    \begin{split}
        S_{12}&\ll  \log T \sum\limits_{(T-H - 1) \leq \gamma' \le (T+ 1)} \  \int\limits_{T-H}^T  \big |h[(t-\gamma')\log x]\big |  \, \d t\\
        &\ll  H\log^2 T.
    \end{split}
\end{equation*}
To balance these two error terms,
we choose $H=\sqrt{T}$.  Therefore, we conclude that
\begin{align}\label{R1Rep1Step7}
S_1 
\ll \sqrt{T}\log^2T.\nonumber
\end{align}

We estimate $S_2$ similarly, by splitting the range of integration from 1 to $H$ and from $H$ to T. We again find that
\begin{align*}
S_2\ll \sqrt{T}\log^2T.
\end{align*}
By combining the estimates for $S_1$ and $S_2$, the fourth error term on the right-hand side of \eqref{R1Rep1Step1} can be estimated as
\begin{equation}\label{R1Rep1Step8}
\log T \int\limits_1^T\Big[\tfrac{1}{T-t+1} + \tfrac{1}{t+1} \Big] \hspace{-.3cm}\sum\limits_{|t-\gamma'|\le \tfrac{1}{\log x}} \hspace{-.3cm}\big |h[(t-\gamma')\log x]\big |  \, \d t\ll \sqrt{T}\log^2T.
\end{equation}

For the first error term of \eqref{R1Rep1Step1}, we again split the range of integration and find that
\begin{equation*}\label{R1Rep1Step9}
\begin{split}
 \int\limits_1^T \sum\limits_{\gamma \in I} \big |h[(t-\gamma)\log x]\big | \hspace{-.4cm} \sum\limits_{|t-\gamma'|\le \tfrac{1}{\log x}}\hspace{-.5cm}  \big |h[(t-\gamma')\log x]\big | \, \d t &=\int\limits_1^{T-1} \sum\limits_{\gamma \in I} \big |h[(t-\gamma)\log x]\big | \hspace{-.4cm} \sum\limits_{|t-\gamma'|\le \tfrac{1}{\log x}}\hspace{-.5cm}  \big |h[(t-\gamma')\log x]\big | \, \d t \\
&\quad +  \int\limits_{T-1}^T\sum\limits_{\gamma \in I} \big |h[(t-\gamma)\log x]\big | \hspace{-.4cm} \sum\limits_{|t-\gamma'|\le \tfrac{1}{\log x}}\hspace{-.5cm}  \big |h[(t-\gamma')\log x]\big | \, \d t \\
&= \Sigma_1 + \Sigma_2.
\end{split}
\end{equation*}
For $\gamma \in I$ and $t\in [1,T-1]$, we know $h[(t-\gamma)\log x] \ll \tfrac{1}{(t-\gamma)^2 \log^2 x}$. Since $h \in L^1$, by an argument similar to the proof of \eqref{eq:SumOverZerosEst2}, we see 
\begin{align*}
\Sigma_1 &\ll  \int\limits_1^{T-1} \sum\limits_{\gamma \in I } \frac{1}{(t-\gamma)^2\log^2 x} \hspace{-.3cm} \sum\limits_{|t-\gamma'|\le \tfrac{1}{\log x}}\hspace{-.3cm} \big |h[(t-\gamma')\log x]\big | \, \d t  \nonumber\\
&\ll  \int\limits_1^{T-1} \bigg[\frac{1}{T-t+1} + \frac{1}{t+1} \bigg] \log T \hspace{-.3cm}\sum\limits_{|t-\gamma'|\le \tfrac{1}{\log x}}\hspace{-.3cm} \big |h[(t-\gamma')\log x]\big | \, \d t  \nonumber\\
&\ll  \sqrt{T}\log^2 T,\nonumber
\end{align*}
where we used \eqref{R1Rep1Step8} in the last line.
Since $T \in \{T_n\}$, we know that $|\gamma - T| \gg \tfrac{1}{\log T}$. 
Thus for $t \in I$, we have that $T-1 \leq t\leq T$ and $T < \gamma \leq T+ \tfrac{1}{\log x}$ imply 
$|t- \gamma| \gg \tfrac{1}{\log T}.$
Since $|I|<1$ and $x\ge 4$, using that $h(v)\ll \frac{1}{v^2}$ for all $|v|>0$, we know that
\begin{align}
\sum\limits_{\gamma \in I } \big |h[(t-\gamma)\log x]\big | \ll \sum\limits_{\gamma \in I } \left|h\left(\tfrac{\log x}{\log T}\right)\right| \ll  \log^2 T \sum\limits_{\gamma \in I } 1  \ll  \log^3 T.\nonumber
\end{align}
Hence, since $\gamma'$ is contained in an interval of size less than $1$, it follows that
\begin{align*}\label{R1Rep1Step12}
\Sigma_2 
&\ll \log^3 T  \int\limits_{T-1}^T\sum\limits_{|t-\gamma'|\le \tfrac{1}{\log x}}\big |h[(t-\gamma')\log x]\big | \, \d t  \nonumber\\
&\ll \frac{\log^3 T}{\log x}\sum\limits_{|T-\gamma'|\le 2}\ \int\limits_{-\infty}^{\infty}  |h(u)|  \, \d u\nonumber\\
&\ll \frac{\log^4 T}{\log x} \nonumber
\end{align*}
for all $T\in \{T_n\}$. Hence combining our estimates for $\Sigma_1$ and $\Sigma_2$ gives
\begin{equation*}\label{R1Rep1Step15}
\int\limits_1^T \sum\limits_{\gamma \in I } \big |h[(t-\gamma)\log x]\big |\hspace{-.4cm} \sum\limits_{|t-\gamma'|\le \tfrac{1}{\log x}}\hspace{-.4cm} \big |h[(t-\gamma')\log x]\big | \, \d t =  \Sigma_1 + \Sigma_2 \ll \sqrt{T}\log^2 T.
\end{equation*}
Therefore, $R_1$ is confined to $\gamma,\gamma' \in [0,T]$ with an added error of $O(\sqrt{T} \log^2 T)$.  Similarly, we may extend the integral range of $[1,T]$ to $(-\infty, \infty)$ with the same error. Thus, 

\begin{equation*}\label{R1Rep1Step16}
R_1=\sum \limits_{0< \gamma,\gamma'\leq T} \int\limits_{-\infty}^{\infty} h[(t-\gamma)\log x] \ h[(t-\gamma')\log x] \d t +O(\sqrt{T} \log^2 T).
\end{equation*}
\indent We now use the properties of $h(v)$ expressed in Lemma \ref{lem:functions} to simplify our expression for $R_1$. Since $h\in L^1$ and it is even, we can use the substitution $u=(t-\gamma')\log x$ together with convolution to find that
\begin{align}\label{R1Rep1Step17}
R_1
&=\frac{1}{\log x} \sum \limits_{0< \gamma,\gamma'\leq T} \int\limits_{-\infty}^{\infty} h(v-u) \ h(u) \d u +O(\sqrt{T} \log^2 T) \nonumber\\
&=\frac{1}{\log x} \sum \limits_{0< \gamma,\gamma'\leq T} h\ast h(v) +O(\sqrt{T} \log^2 T),\nonumber
\end{align}
with $v=(\gamma - \gamma')\log x$. Since $h \in L^1$, we know that convolution is well-defined and $\widehat{h\ast h}= \hat{h}^2$. Furthermore, from Lemma \ref{lem:functions}, we know that $\widehat{h}\in L^2$, and therefore $k(\xi)=\tfrac{1}{\pi^2}\widehat{h}(\xi)^2\in L^1$. Thus by Lemma \ref{lem:functions}, \eqref{eq:k_def}, and the properties of Fourier Transform, we have
\begin{equation*}\label{R1Rep1Final}
R_1 
= \frac{\pi^2}{\log x}  \sum \limits_{0<\gamma, \gamma' \leq T} \hspace{-.2cm}  \hat{k}[(\gamma - \gamma')\log x] + O(\sqrt{T} \log^2 T),\nonumber
\end{equation*}
as claimed.\vspace{.4cm}

\noindent\textit{Proof of part (b).}
The proof here is similar, but we highlight some important differences.
Recall that 
\begin{equation*}
    R_2=\int\limits_0^T [B(t+\Delta)-B(t)]^2 \, \d t,
\end{equation*}
where we defined $B(t)$ in \eqref{eq:A_B_defs}. 
First, by Lemma \ref{lem:a(T)} and part (a) of Lemma \ref{lem:r}, since $\Delta\ll T^b$ with $b<\tfrac{1}{2}$, we have that 
\begin{equation*}
    \int\limits_{1+\Delta}^{T+\Delta} B(t)^2\, \d t = \int\limits_{1}^{T} B(t)^2\, \d t + O\!\left(T\sqrt{\frac{\log\log T}{\log T}}\right).
\end{equation*}
Therefore, we find that 
\begin{equation*}\label{eq:r2-and-r1}
    R_2 = 2R_1 -2R_{22} + O\!\left(T\sqrt{\frac{\log\log T}{\log T}}\right),
\end{equation*}
where
\begin{equation}\label{eq:r22}
    R_{22}:=\int\limits_0^T \sum_{\gamma,\,\gamma'}h[(t+\Delta-\gamma)\log x]h[(t-\gamma')\log x]\, \d t.
\end{equation}
As in part (a), we restrict the double sum in \eqref{eq:r22} to the interval $[0, T]$ and then extend the integral to $\R$, up to an error term $o(T)$. For this purpose, note that 
\begin{equation*}\label{eq:SumOverZerosEst22}
\sum\limits_{\substack{\gamma \\ \gamma \notin [0,T]}} h[(t+\Delta-\gamma)\log x] =  \sum\limits_{\gamma \in I_\Delta }  h[(t+\Delta-\gamma)\log x]  + O\Big( \Big[\tfrac{1}{T-t+1} + \tfrac{1}{t+1} \Big] \log T \Big),
\end{equation*}
where $I_\Delta =\{\gamma  :  T< \gamma \leq T  + \Delta+  \tfrac{1}{\log x} \}$. Note that $|I_\Delta|\ll (\Delta+1)\log T$. By computations similar to those of part (a), we find that
 \begin{align}\label{eq:r22-firstErrors}
R_{22} 
 &=\!\!\!\!\!\!\!\sum \limits_{0< \gamma,\gamma'\leq T} \int\limits_{1}^{T} h[(t+\Delta-\gamma)\log x] \, h[(t-\gamma')\log x]\, \d t + O\bigg(\frac{(\Delta+1)\log^2 T}{\log x} \bigg) + O\bigg( \sqrt{T}\log^2 T \bigg)
 \nonumber\\
&  \hspace{-.1cm} \quad +   O\!\left(\int\limits_1^T \left|\sum\limits_{\gamma \in I_\Delta} h[(t+\Delta-\gamma)\log x] \hspace{-.4cm} \sum\limits_{|t-\gamma'|\le \tfrac{1}{\log x}}\hspace{-.5cm}  h[(t-\gamma')\log x]\right| \, \d t \right).
\end{align}
The next step is different from the steps in the proof of part (a). To bound the last error term in \eqref{eq:r22-firstErrors}, we use the Cauchy-Schwarz inequality:
\begin{align}
    \int\limits_1^T \left|\sum\limits_{\gamma \in I_\Delta} h[(t+\Delta-\gamma)\log x] \hspace{-.4cm} \sum\limits_{|t-\gamma'|\le \tfrac{1}{\log x}}\hspace{-.5cm}  h[(t-\gamma') \log x]\right| \, &\d t \nonumber \\
    &\hspace{-3.5cm}\le \Bigg\|\sum\limits_{\gamma \in I_\Delta} h[(t+\Delta-\gamma)\log x] \Bigg\|_2
    \cdot\Bigg\| \sum\limits_{|t-\gamma'|\le \tfrac{1}{\log x}}\hspace{-.5cm}  h[(t-\gamma')\log x]\Bigg\|_2\nonumber\\
    &\hspace{-3.5cm}= J_1 \cdot J_2. \label{eq:r22-cs}
\end{align}
To estimate $J_1$, we expand the integral, apply Cauchy-Schwarz once more, and use that $h\in L^2$ and $|I_\Delta|\ll (\Delta+1)\log T$. This gives
\begin{align}
    J_1^2 &= \sum\limits_{\gamma,\,\gamma' \in I_\Delta}\int\limits_0^Th[(t+\Delta-\gamma)\log x]h((t+\Delta-\gamma')\log x)\, \d t\nonumber\\
    &\ll \frac{(\Delta+1)^2\log^2 T}{\log x}.\label{eq:r22-j1}
\end{align}
To estimate $J_2$, we use \eqref{eq:SumOverZerosEst1} to extend the sum over zeros to the interval $[0, T+1]$, together with the bound $N(T+1)\ll T\log T$ and the fact that $h\in L^1$. This yields 
\begin{align}\label{eq:r22-j2}
    J_2^2 &= \sum_{0<\gamma,\, \gamma'\le T+1}\int_1^T h[(t-\gamma)\log x]h[(t-\gamma')\log x]\ \d t +O\!\left(T\log^2 T\right)\nonumber\\
    &= R_1 + O(T\log^2 T)\nonumber\\
    &\ll T\log^2 T,
\end{align}
where we used part (a) of Lemma \ref{lem:r}. Combining \eqref{eq:r22-firstErrors}, \eqref{eq:r22-cs}, \eqref{eq:r22-j1}, and \eqref{eq:r22-j2}, we obtain
\begin{equation*}
    R_{22} 
 =\sum \limits_{0< \gamma,\gamma'\leq T} \int\limits_{1}^{T} h[(t+\Delta-\gamma)\log x] \, h[(t-\gamma')\log x]\, \d t + O\bigg((\Delta+1)\sqrt{T}\log^2 T \bigg).
\end{equation*}
Similarly, the integral above may be extended to $\R$ up to the same error term. The rest of the proof is analogous to part (a). 
\end{proof}

\subsection{A modified pair correlation approach}
The next step is to introduce the weight function $w(u)$, from \eqref{FFunctionAlphaT}, to write $R_1$ and $R_2$ in Lemma \ref{lem:unbounded_discontinuities} in terms of Montgomery's function $F(\alpha)$ and Chan's function $F_\Delta(\alpha)$. 

\begin{lemma}\label{lem:R_expression}
Let $T\in \{T_n\}$, where $T_n$ satisfies \eqref{TnSequence}. Define $k$ as in \eqref{eq:k_def}, and assume RH. For $4\le x\le T$, and $0<\Delta\le T$, we have
\begin{align*}
     &(\textup{a})\  R_1=\frac{\pi^2}{\log x}\sum_{0<\gamma,\, \gamma'\le T}\widehat{k}[(\gamma-\gamma')\log x]w(\gamma-\gamma') + O\!\left(\frac{T\log^2 T}{\log^3  x}\right);\\
     &(\textup{b})\  R_2 = \frac{2\pi^2}{\log x}\sum_{0 < \gamma, \gamma' \le T} \left\{\widehat{k}[(\gamma - \gamma' )\log x]w(\gamma-\gamma')-  \widehat{k}\left[\left(\gamma - \gamma' -\Delta \right)\log x\right]w\!\left(\gamma - \gamma' -\Delta \right) \right\} \\
     &\quad \quad \quad \quad +O\!\left(\frac{T\log^2 T}{\log^3  x}\right).
\end{align*}
\end{lemma}\begin{proof}
The proofs of the expressions in parts (a) and (b) are proved using similar methods, but the proof of part (b) is more involved. For this reason, we only work out part (b). Recall that $k$ is the function defined in \eqref{eq:k_def}. We have that $\widehat{k}(y)\ll\min ( 1, \tfrac{1}{y^2}).$ From this estimate we introduce the weight function $w(u)$, defined in \eqref{FFunctionAlphaT}, into the sum over zeros
\[ \sum_{0 < \gamma, \gamma' \le T}\hspace{-.1cm}  \widehat{k}[(\gamma - \gamma' )\log x]\] 
 using the following argument.  We consider the difference
\begin{equation*}
D:= \sum \limits_{0 < \gamma, \gamma' \leq T} \left\{ \widehat{k}\Big[\Big(\gamma - \gamma' -\Delta \Big)\log x\Big]-\widehat{k}\Big[\Big(\gamma - \gamma' -\Delta \Big)\log x\Big]w\Big(\gamma - \gamma' -\Delta \Big) \right\}.
\end{equation*}
Using the facts that $N(T) \ll T \log T$, there are $O(\log t)$ zeros in any given interval $[t,t+1]$, and that $\Delta \le T$, we have
\begin{align*}\label{R1Rep2Step3}
 D   
&\ll  \frac{1}{\log^2 x} \sum \limits_{0 <  \gamma' \leq T} \sum \limits_{\gamma}  \frac{1}{4+(\gamma-\gamma'- \Delta )^2} \nonumber\\   
&\ll  \frac{1}{\log^2 x} \sum \limits_{0 <  \gamma' \leq T}   \log
\!\left(\gamma' +\Delta\right) \ll \frac{T\log^2T}{\log^2 x}.\nonumber
\end{align*}
Therefore,
\begin{equation*}\label{R1Rep2Step4}
 \sum \limits_{0 < \gamma, \gamma' \leq T} \widehat{k}\Big[\big(\gamma - \gamma' -\Delta \big)\log x\Big] =  \sum \limits_{0 < \gamma, \gamma' \leq T} \hat{k}\left[\big(\gamma - \gamma' -\Delta \big)\log x\right] w\Big(\gamma - \gamma' -\Delta \Big) + O\left(\frac{T\log^2T}{\log^2 x}\right).
\end{equation*}
Similarly, we may introduce the weight $w(u)$ into the the other terms in the representations of $R_1$ and $R_2$ in Lemma \ref{lem:unbounded_discontinuities} to complete the proof.
\end{proof}
\noindent Using Lemma \ref{lem:R_expression} and the properties of $F(\alpha)$ and $F_\Delta(\alpha)$, we take $x=T^\beta$ and proceed to estimate $R_i$. 
\begin{lemma}[Estimates of $R_i$]\label{lem:r} Assume RH.
Let $T\in \{T_n\}$, where $T_n$ satisfies \eqref{TnSequence}. Fix $0<\beta\le 1$, let $g$ be defined in \eqref{eq:def_g}, and define $R_i$ as in \eqref{eq:GHR_defs}. For $T\ge 4$, $x=T^\beta$, and $0<\Delta= o(\log^2 T)$, we have 
\begin{align*}
     &(\textup{a})\  R_1=\frac{T}{2} \left\{\int\limits_1^\infty \frac{F(\alpha)}{\alpha^2}\, \d \alpha + 
    \int\limits_0^1 v\, g^2(v)\,\d v+ \frac{g(0)^2}{2\beta^2} -\log \beta
    \right\} + o(T).\\
     &(\textup{b})\  R_2 = T \left\{ \int\limits_{0}^{1} v \, g^2\left( v \right) \left[1-w\!\left(\Delta\right)\cos(\Delta v \beta \log T)\right] \d v -\log(\beta) \right. \nonumber\\
 &\quad \quad \quad \quad    - \left. w\!\left(\Delta\right)\int\limits_{\Delta\beta\log T}^{\Delta\log T} \frac{ \cos(u) }{u} \d u + \frac{1}{2}\int\limits_{1}^{\infty} \frac{2F(\alpha) -F_{\Delta}(\alpha)-F_{-\Delta}(\alpha)}{\alpha^2}\, \d \alpha \right\} +o(T), \nonumber 
\end{align*}
where the error term on part (a) is of size $O\!\left(T\sqrt{\frac{\log \log T}{\log T}} \right)$, and the error term on part (b) is of size $O\!\left(\frac{T}{\sqrt{\log T}}\right)+O\!\left(\frac{\Delta}{\log^2 T}\right).$ 
\end{lemma}
\begin{proof}[Proof of part (a)]

Recall the definition of the function $F(\alpha)$ and $w(u)$ in \eqref{FFunctionAlphaT}.  Then using the definition of Fourier transform, we manipulate the sum over zeros in the representation formula for $R_1$ in Lemma \ref{lem:R_expression} to yield
\begin{align}\label{KSumWithFAlpha}
\sum_{0<\gamma,\, \gamma'\le T}\widehat{k}[(\gamma-\gamma')\log x]w(\gamma-\gamma') 
&=  \int\limits_{-\infty}^{\infty} k(u)  \sum \limits_{0 < \gamma, \gamma' \leq T}  e^{-2\pi i u(\gamma - \gamma')\log x}  w(\gamma - \gamma')\, \du   \nonumber\\
&=\frac{T\log T}{(2\pi)^2\beta }  \int\limits_{-\infty}^{\infty}k\Big( \frac{\alpha}{2\pi \beta} \Big) F(\alpha)\, \d\alpha.
\end{align}
Then, inputting \eqref{KSumWithFAlpha} into part (a) of Lemma \ref{lem:R_expression} gives
\begin{equation}\label{R1Rep1}
R_1= \frac{\pi^2}{\log x} \hspace{-.1cm}\sum_{0<\gamma,\, \gamma'\le T} \hspace{-.3cm}\widehat{k}[(\gamma-\gamma')\log x]w(\gamma-\gamma') + O\Big( \frac{T}{\log T}\Big) =\frac{T}{(2 \beta)^2}\hspace{-.1cm} \int\limits_{-\infty}^{\infty} \hspace{-.1cm} k\Big( \frac{\alpha}{2\pi \beta} \Big) F(\alpha) \, \d\alpha + O\Big( \frac{T}{\log T}\Big). 
\end{equation}
Recall from \eqref{eq:k_def} that $k(u)$ is piecewise defined with a transition at 
$u= 1/(2\pi)$.  Thus, we use \eqref{eq:F_asymptotic} and the fact that $F(\alpha)$ and $k(u)$ are both even and nonnegative functions to rewrite the above integral over $k$ and $F$ as 
\begin{align}\label{RResult3Step1}
\int\limits_{-\infty}^{\infty}k\Big( \frac{\alpha}{2\pi \beta} \Big) F(\alpha)\, \d\alpha 
&= 2 \int\limits_{0}^{\beta}k\Big( \frac{\alpha}{2\pi \beta} \Big) \left[ \alpha + o(1) + T^{-2\alpha} \log T (1+o(1))\right]\d\alpha\nonumber\\
&\quad+ 2\int\limits_{\beta}^{1} \left( \frac{\beta}{\alpha}\right)^2 \left[ \alpha + o(1) + T^{-2\alpha} \log T (1+o(1))\right]\d\alpha + 2\int\limits_{1}^{\infty} \left( \frac{\beta}{\alpha}\right)^2  F(\alpha)\, \d\alpha.
\end{align}
For the second integral on the right-hand side in \eqref{RResult3Step1}, because $\beta$ is fixed, we know that
\begin{equation}\label{RResult3Step2}
2\int\limits_{\beta}^{1} \left( \frac{\beta}{\alpha}\right)^2 \left[ \alpha + o(1) + T^{-2\alpha} \log T (1+o(1))\right]\d\alpha = - 2\beta^2\log \beta + o(1). 
\end{equation}
To compute the first integral on the right-hand side of \eqref{RResult3Step1}, we use the facts that $k\!\left(\frac{\alpha}{2\pi\beta}\right)= g^2\!\left(\frac{\alpha}{\beta}\right)$ for $0\le \alpha \le \beta$, that $0<\beta\le 1$ is fixed, and that $k$ is smooth near the origin and uniformly bounded. By technical yet straightforward  manipulations, we find that 
\begin{equation}\label{RResult3Step8}
\begin{split}
2 \hspace{-.1cm} \int\limits_{0}^{\beta} \hspace{-.1cm} k\Big( \frac{\alpha}{2\pi \beta} \Big) \hspace{-.1cm}\left[ \alpha + o(1) +T^{-2\alpha} \log T  (1+o(1))\right]\hspace{-.1cm} \,\d\alpha  
=2\beta^2 \hspace{-.1cm} \int\limits_{0}^{1} \hspace{-.1cm}  v\, g^2( v)\,\d v +g^2(0) + o(1).
\end{split}
\end{equation}
Combining the estimates \eqref{RResult3Step2} and \eqref{RResult3Step8} yields
\begin{equation}\label{RResult3Step9}
\int\limits_{-\infty}^{\infty} k\Big( \frac{\alpha}{2\pi \beta} \Big)  F(\alpha)\,\d\alpha = 2\beta^2  \int\limits_{0}^{1} v \, g^2( v)\,\d v   - 2\beta^2\log \beta  +  g^2(0) + 2\beta^2 \int\limits_{1}^{\infty} \frac{F(\alpha)}{\alpha^2}\, \d\alpha + o(1).
\end{equation}
Inputting \eqref{RResult3Step9} into the representation for $R_1$ in \eqref{R1Rep1} concludes the proof of part (a). \vspace{.4cm}

\noindent\textit{Proof of part (b).}
We consider the definition of the function $F_{\Delta}(\alpha)$ in \eqref{eq:F_delta_def}. 
Then using the definition of Fourier transform, we manipulate the sum over zeros in the representation formula for $R_2$ in Lemma \ref{lem:R_expression} to yield
\begin{align}\label{KSumWithFDeltaAlpha}
\sum \limits_{0 < \gamma, \gamma' \leq T} &\hat{k}\left(\big(\gamma - \gamma' -\Delta \big)\log x\right)  w\Big(\gamma - \gamma' -\Delta \Big)\nonumber\\
&= \int\limits_{-\infty}^{\infty} k(u)  \sum \limits_{0 < \gamma, \gamma' \leq T}  T^{-2\pi \beta i u(\gamma - \gamma'-\Delta)}  w(\gamma - \gamma'-\Delta)\,\d u  \nonumber\\
&=\frac{T\log T}{(2\pi)^2\beta }  \int\limits_{-\infty}^{\infty}k\Big( \frac{\alpha}{2\pi \beta} \Big) F_{\Delta}(\alpha)\,\d\alpha.
\end{align}
\noindent Then, inputting \eqref{KSumWithFAlpha} and \eqref{KSumWithFDeltaAlpha} into part (b) of Lemma \ref{lem:R_expression} gives
\begin{equation*}\label{R1Rep2StepFinal}
\begin{split}
R_2
&= \frac{T}{2 \beta^2} \left\{\int\limits_{-\infty}^{\infty}k\Big( \frac{\alpha}{2\pi \beta} \Big) F(\alpha)\, \d \alpha -\int\limits_{-\infty}^{\infty}k\Big( \frac{\alpha}{2\pi \beta} \Big) F_{\Delta}(\alpha)\, \d\alpha\right\} + O\!\left( \frac{T}{\log T}\right).
\end{split}
\end{equation*}
By splitting the second integral above using the symmetry relations for $F_{\Delta}$ in \eqref{eq:symmetry_rels}, we have
\begin{equation*}\label{R2ResultStep1}
\begin{split}
- \frac{T}{2 \beta^2}\int\limits_{-\infty}^{\infty}k\Big( \frac{\alpha}{2\pi \beta} \Big) F_{\Delta}(\alpha) \,\d\alpha = -\frac{T}{2 \beta^2} \int\limits_0^\infty k\Big( \frac{\alpha}{2\pi \beta} \Big)[F_\Delta(\alpha)+F_{-\Delta}(\alpha)] \, \d \alpha.
\end{split}
\end{equation*}
Next, we divide the integral over the intervals $(0, \, \beta)$, $(\beta, 
\, 1)$, and $(1, \, \infty)$, and apply \eqref{eq:F_delta_asymptotic}. Consequently, since $k(u)$ is even, $k(0) = g^2(0)$, and $T^{i\alpha\Delta}+T^{-i\alpha\Delta}=2\cos(\Delta\alpha\log T)$, we obtain
\begin{equation*}\label{R2ResultStep2}
\begin{split}
   - \frac{T}{2\beta^2}\int\limits_0^\beta & k\Big( \frac{\alpha}{2\pi \beta} \Big) [F_\Delta(\alpha)+F_{-\Delta}(\alpha)]\, \d \alpha = -\frac{Tg^2(0)}{2\beta^2} - T\, w(\Delta)\int\limits_0^1 v \, g(v)^2\cos (\Delta v \beta \log T )\, \d v + o(T), 
\end{split}
\end{equation*}
and 
\begin{equation*}\label{R2ResultStep3}
   - \frac{T}{2\beta^2}\int\limits_\beta^1 k\Big( \frac{\alpha}{2\pi \beta} \Big)[F_\Delta(\alpha)+F_{-\Delta}(\alpha)]\, \d \alpha = -T\, w(\Delta)\int\limits_{\Delta \beta\log T }^{\Delta\log T}\frac{\cos u}{u}\, \d u +o(T).
\end{equation*}
By combining the above integrals, we have that 
\begin{align}\label{R2ResultStep4}
- \frac{T}{2 \beta^2} \int\limits_{-\infty}^{\infty}k\Big( \frac{\alpha}{2\pi \beta} \Big) F_{\Delta}(\alpha)\, \d\alpha &=  -T\left\{\frac{g^2(0) }{2\beta^2} + w(\Delta)\int\limits_0^1 v \, g(v)^2\cos (\Delta v \beta\log T)\, \d v \right. \nonumber\\
&\quad  \left. + \, w(\Delta)\int\limits_{\Delta \beta\log T }^{\Delta\log T }\frac{\cos u}{u}\, \d u +\frac12 \int\limits_{1}^{\infty} \frac{F_{\Delta}(\alpha)+F_{-\Delta}(\alpha)}{\alpha^2}\,\d \alpha\right\} +o(T).
\end{align}
From the proof of part (a) (see \eqref{RResult3Step9}), we know that 
\begin{equation}\label{R2ResultStep5}
\begin{split}
\frac{T}{2\beta^2} \int \limits_{-\infty}^{\infty} k\Big( \frac{\alpha}{2\pi \beta} \Big) F(\alpha)  \, \d\alpha = T \left[  \int\limits_{0}^{1} v \, g^2\left( v \right)\,\d v -\log \beta +\frac{g^2(0)}{2\beta^2} + \int\limits_{1}^{\infty} \frac{F(\alpha)}{\alpha^2}\, \d\alpha  \right] +o(T).
\end{split}
\end{equation}
By adding \eqref{R2ResultStep4} and \eqref{R2ResultStep5} together, our asymptotic formula for $R_2$ reduces to  
\begin{align*}\label{R2ResultFinal}
R_2 = &T \left[ \int\limits_{0}^{1} v \, g^2\left( v \right) \left(1-w(\Delta)\cos(\Delta v \beta\log T)\right) \d v -\log \beta \right. \nonumber\\
 &- \left.w(\Delta)\int\limits_{\Delta \beta\log T }^{\Delta\log T} \frac{ \cos u }{u} \d u + \frac{1}{2}\int\limits_{1}^{\infty} \frac{2F(\alpha) -F_{\Delta}(\alpha)-F_{-\Delta}(\alpha)}{\alpha^2}\d \alpha \right] +o(T),\nonumber 
\end{align*} 
which completes the proof.
\end{proof}

\section{Contributions from the primes}\label{sec:primes}
In this section, we estimate the expressions $G_i+H_i$. First, we obtain intermediate expressions for $G_i$ and $H_i$ separately.
\subsection{Expressions for $G_i$ and $H_i$}
We begin with a useful Lemma that helps estimate the second moment of some trigonometric polynomials.

\begin{lemma}
Let $T>0$, and let  $\{a_n\}_{n\ge1}$ and $\{h_n\}_{n\ge1}$ be sequences of real numbers such that \\ $\sum_{n=1}^\infty (n|a_n|^2+|a_n|) <\infty.$ Denote
\[C:= \sum_{n\ge 2}|a_n|\sum_{m\ge 2} \frac{|a_m|}{\log m}.
\]
Then,
\begin{align*}
    &\textup{(a)} \int\limits_0^T \left| \sum_{n=1}^\infty a_n\,\cos (t\log n)
    \right|^2 \d t = T\,a_1^2 + \frac{T}{2}\sum_{n\ge 2}a_n^2 + O\!\left(C+\sum_{n=1}^\infty n|a_n|^2 \right) 
    \\
    &\textup{(b)}  \int\limits_0^T \left| \sum_{n=1}^\infty a_n\,\sin (t\log n)
    \right|^2 \d t = \frac{T}{2}\sum_{n\ge 2}a_n^2+ O\!\left(C+\sum_{n=1}^\infty n|a_n|^2 \right)  
    \\
    &\textup{(c)}  \int\limits_0^T \left| \sum_{n=1}^\infty a_n\left[\,\cos ((t+h_n)\log n)\!-\!\cos(t\log n)\right]
    \right|^2 \d t = T\sum_{n\ge 2}a_n^2 \, [1\!-\!\cos(h_n\log n)]+ O\!\left(C+\sum_{n=1}^\infty n|a_n|^2 \right)  
    \\
    &\textup{(d)} \int\limits_0^T \left| \sum_{n=1}^\infty a_n\left[\,\sin ((t+h_n)\log n)\!-\!\sin(t\log n)\right]
    \right|^2 \d t = T\sum_{n\ge 2}a_n^2 \, [1\!-\!\cos(h_n\log n)]+ O\!\left(C+\sum_{n=1}^\infty n|a_n|^2 \right). 
\end{align*}
The implied constants are universal.
\end{lemma}
\begin{proof}
A classical result of Montgomery and Vaughan \cite[Corollary 3]{MV3} states that, for complex numbers $\{b_n\}_{n\ge1}$, we have
\begin{equation}\label{eq:montgomery-vaughan}
    \int\limits_0^T \left| \sum_{n=1}^\infty b_n\,n^{-it}
    \right|^2 \d t = \sum_{n= 1}^\infty|b_n|^2\, \big(T+O(n)\big).
\end{equation}
For part (a), let $z:=\sum_{n=1}^\infty a_n\,n^{-it}$, and note that $\re z = \sum_{n=1}^\infty a_n\,\cos (t\log n)$. Consider the identity 
\begin{equation}\label{eq:complex-id1}
    (\re z)^2 = \frac{|z|^2+\re(z^2)}{2}.
\end{equation}
By \eqref{eq:montgomery-vaughan}, we have 
\begin{equation}
    \int\limits_0^T \frac{|z|^2}{2}\,\d t = \frac{1}{2}\sum_{n\ge 1}a_n^2\, \big(T+O(n)\big).\label{eq:|z|^2/2}
\end{equation}
We write $z=a_1 + \sum_{n=2}^\infty a_n\,n^{-it}$ and use that, for $n\ge 2,$ we have $\int\limits n^{-it}\d t = in^{-it}/\log n$. This yields
\begin{align}
     \int\limits_0^T \frac{z^2}{2}\,\d t &= \frac{a_1^2}{2}\,T + O\!\left( 
     |a_1| \sum_{n \ge 2}\frac{|a_n|}{\log n}\right) + O\!\left(\sum_{n,\,m\ge 2}\frac{|a_na_m|}{\log (mn)}\right)\label{eq:z^2/2}
     \nonumber\\
     &= \frac{a_1^2}{2}\,T + O\!\left(C+\sum_{n=1}^\infty n|a_n|^2 \right)
    .
\end{align}
Here, we used the Cauchy-Schwarz inequality to obtain \[|a_1|\sum_{n \ge 2}\frac{|a_n|}{\log n}\le |a_1|\sqrt{\sum_{n \ge 2}\frac{1}{n\log^2 n}} \cdot \sqrt{\sum_{n \ge 2}n|a_n|^2} \ll \sum_{n = 1}^\infty n|a_n|^2.\]
Combining \eqref{eq:complex-id1},\eqref{eq:|z|^2/2}, and \eqref{eq:z^2/2}, we obtain part (a). Part (b) is analogous, using the identity
\begin{equation*}
    (\im z)^2 = \frac{|z|^2-\re(z^2)}{2}
\end{equation*}
in place of \eqref{eq:complex-id1}. Note that, since $\sin 0=0,$ part (a) has an extra contribution from the term $a_1$ that is not present in part (b). For parts (c) and (d), we use the identities
\begin{equation*}
    \cos ((t+h_n)\log n)-\cos(t\log n)=\re[n^{-it}(n^{-ih_n} -1)],
\end{equation*}
\begin{equation*}
 \sin ((t+h_n)\log n)-\sin(t\log n)=-\im[n^{-it}(n^{-ih_n} -1)], 
\end{equation*}

and
\begin{equation*}
    \ |n^{-ih_n} -1|^2 = 2(1-\cos(h_n\log n)).
\end{equation*}
Then, we apply the same argument above with $z=\sum_{n=1}^\infty a_n(n^{-ih_n} -1)\,n^{-it}$, using Montgomery and Vaughan's result \eqref{eq:montgomery-vaughan} with $b_n=a_n(n^{-ih_n} -1).$ 
\end{proof}
Using the previous lemma, 
we obtain the following expressions for $G_i$.
\begin{lemma}[$G_i$]\label{lem:G}
Let $\Delta>0$ and $4\le x\le T$. Let $G_1$ and $G_2$ be defined in \eqref{eq:GHR_defs}.  Then, 
\begin{align*}
   &\textup{(a)} \  G_1=-\frac{T}{2}\sum_{n\le x} \frac{\Lambda^2(n)}{n\log^2 n}f^2\!\left(\frac{\log n}{\log x}
    \right) + O\!\left(\frac{x}{\log x}\right);\\
    &\textup{(b)} \ G_2 = -T \sum_{n\le x} \frac{\Lambda^2(n)}{n\log^2 n}f^2\!\left(\frac{\log n}{\log x} \right)\left[1-\cos\left(\Delta \log n \right)\right] + O\!\left(\frac{x}{\log x}\right) .
\end{align*}
\end{lemma}
Next, with the goal of studying $H_1$ and $H_2$, we use some estimates of Goldston, together with some trigonometric identities, to obtain expressions for the real and imaginary parts of integrals involving $\log \zeta(1/2+it)$ times trigonometric functions. Some of these results appear previously in \cite{Go} (part (b)) and implicitly in \cite{Fu1} (part (d)). We collect them all in the following lemma, for the reader's convenience.
\begin{lemma}\label{lem:log_zeta_times_trig} Assume RH. 
Let $T>1$, let $h\in \mathbb{R}$, and let $n\ge 2$ be an integer. Denote
\[\mathcal{E}=\mathcal{E}(n,\,T):= n^{1/2}\log\log 3n + \frac{n^{1/2}\log T}{\log n}.
\]
Then, the following estimates hold:
\begin{align*}
    &\textup{(a)} \int\limits_1^T \log |\zeta(\tfrac{1}{2}+it)|\,\cos(t\log n)\,\d t = \frac{T}{2}\frac{\Lambda(n)}{n^{1/2}\log n} + O(\mathcal{E});
    \\
    &\textup{(b)} \int\limits_1^T \pi S(t)\, \sin(t\log n)\,\d t = -\frac{T}{2}\frac{\Lambda(n)}{n^{1/2}\log n} + O(\mathcal{E});
    \\
    &\textup{(c)}  \int\limits_1^T \log |\zeta(\tfrac{1}{2}+it)|\,[\cos((t+h)\log n)+\cos((t-h)\log n)-2\cos(t\log n)]\,\d t = -T\,\frac{\Lambda(n)[1-\cos(h\log n)]}{n^{1/2}\log n} \\
        & \hspace{12.1cm}+ O(\mathcal{E});
    \\
    &\textup{(d)}\int\limits_1^T \pi S(t)\,[\sin((t+h)\log n)+\sin((t-h)\log n)-2\sin(t\log n)]\,\d t = T\,\frac{\Lambda(n)[1-\cos(h\log n)]}{n^{1/2}\log n} + O(\mathcal{E}).
\end{align*}
\end{lemma}
\begin{proof}
Assuming RH, Goldston \cite[p.~169]{Go}, improving upon a contour argument of Titchmarsh \cite{Tit2}, showed that 
\begin{equation}\label{eq:gold_log_times_nit1}
    \int\limits_0^T \log \zeta(\tfrac{1}{2}+it) \, n^{it}\,\d t = \frac{T\,\Lambda(n)}{n^{1/2}\log n} + O(n^{1/2}\log\log 3n) + O\!\left(\frac{n^{1/2}\log T}{\log n}\right)
\end{equation}
and
\begin{equation}\label{eq:gold_log_times_nit2}
    \int\limits_0^T \log \zeta(\tfrac{1}{2}+it)n^{-it}\,\d t \ll \log T.
\end{equation}
Adding \eqref{eq:gold_log_times_nit1} and \eqref{eq:gold_log_times_nit2} and taking real parts yields part (a). Part (b) is \cite[Equation (6.3)]{Go}. 
Parts (c) and (d) follow from parts (a) and (b) after applying the trigonometric identities 
\begin{equation*}
    \cos((t+h)\log n)+\cos((t-h)\log n)-2\cos(t\log n) = -2\cos(t\log n)[1-\cos(h\log n)],
\end{equation*}
and 
\begin{equation*}
    \sin((t+h)\log n)+\sin((t-h)\log n)-2\sin(t\log n) = -2\sin(t\log n)[1-\cos(h\log n)].
\end{equation*}
This completes the proof of the lemma. 
\end{proof}
We now obtain expressions for $H_i$ from the above lemma. As mentioned in the introduction, in the next lemma we use Theorem \ref{thm:variance} to control some of the error terms in part (b), which is the part of the lemma that is relevant to Theorem \ref{thm:short-intervals}.
\begin{lemma}[$H_i$]\label{lem:H} Assume RH.
Let $0<\Delta\le T^b$, with $0<b<1$, and $4\le x\le T$. Let $H_1$ and $H_2$ be defined as in \eqref{eq:GHR_defs}.  Then,
\begin{align*}
    &\textup{(a)}\ H_1= T\sum_{n\le x} \frac{\Lambda^2(n)}{n\log^2 n}f\!\left(\frac{\log n}{\log x}
    \right) +O\!\left(\frac{x\log\log x \log T}{\log^2 x} \right)\\
    &\textup{(b)}\ H_2 = 2T \sum_{n\le x} \frac{\Lambda^2(n)}{n\log^2 n}f\!\left(\frac{\log n}{\log x} \right)\left[1-\cos\left(\Delta \log n \right)\right] +O\!\left(\frac{x\log\log x \log T}{\log^2 x} \right) + O\left(\frac{T}{\sqrt{\log x}}\right).
\end{align*}
\end{lemma}
\begin{proof}
Part (a) follows from part (a) of Lemma \ref{lem:log_zeta_times_trig} and the definition of $H_1$. For part (b), we rearrange the terms and use a change of variables to find that
\begin{equation*}
\begin{split}
     H_2 = &-2\int\limits_1^T \log |\zeta(\tfrac{1}{2}+it)|\left[ A\!\left(t+\Delta\right) + A\!\left(t-\Delta\right) -2A(t)
    \right]\,\d t \\
    &+O\!\left(\int\limits_1^{1+\Delta}|\log |\zeta(\tfrac{1}{2}+it)||\,\left| A\!\left(t\right) - A\!\left(t-\Delta\right)
    \right|\,\d t  \right) \\
    &+ O\!\left(\int\limits_T^{T+\Delta}|\log |\zeta(\tfrac{1}{2}+it)||\,\left| A\!\left(t\right) - A\!\left(t-\Delta\right)
    \right|\,\d t  \right) .
\end{split}
\end{equation*}
By Theorem \ref{thm:variance}, with $0<\Delta\le T^b$, we have
\begin{equation}\label{eq:th1-application}
    \int\limits_T^{T+\Delta}\log^2|\zeta(\tfrac{1}{2}+it)|\,\d t = o(T).
\end{equation}
Note that we used Lemma \ref{lem:a(T)} to show that the contribution from the constant $a$ in Theorem \ref{thm:variance} is $o(T)$ over the interval $[T, T+\Delta]$. By Montgomery and Vaughan's result in \eqref{eq:montgomery-vaughan}, we also have
\begin{equation}\label{eq:A-short-interval}
 \int\limits_T^{T+\Delta}\left| A\!\left(t\right) - A\!\left(t-\Delta\right)
    \right|^2 \, \d t \ll \sum_{n\le x} \frac{\Lambda^2(n)}{n\log^2 n}\big(\Delta + n\big) \ll \Delta \log \log x + \frac{x}{\log x}.
\end{equation}
We use the Cauchy-Schwarz inequality, \eqref{eq:th1-application}, and \eqref{eq:A-short-interval} to obtain
\begin{equation*}
    \int\limits_T^{T+\Delta}|\log |\zeta(\tfrac{1}{2}+it)||\,\left| A\!\left(t\right) - A\!\left(t-\Delta\right)
    \right| \, \d t \ll \sqrt{\Delta T \log \log x} + \sqrt{\frac{Tx}{\log x}}\ll \frac{T}{\sqrt{\log x}},
\end{equation*}
since we have $4\le x\le T$ and $0<\Delta \le T^b.$ The first error term may be treated similarly. This yields
\begin{equation*}
    H_2 = -2\int\limits_1^T \log |\zeta(\tfrac{1}{2}+it)|\left[ A\!\left(t+\Delta\right) + A\!\left(t-\Delta\right) -2A(t)
    \right]\,\d t + O\!\left(\frac{T}{\sqrt{\log x}}\right).
\end{equation*}
The conclusion now follows from part (c) of Lemma \ref{lem:log_zeta_times_trig} and the definition of $A(t)$ in \eqref{eq:A_B_defs}.
\end{proof}
\subsection{Estimating $G_i+H_i$}
Starting from the previous results, we proceed to estimate $G_i+H_i$ asymptotically, taking advantage of some cancellations between their sums via the function $g$ (introduced in \eqref{eq:def_g}). In this section, we diverge from the strategies of previous work of Fujii to obtain more precise input from the primes, which is necessary for Theorem \ref{thm:short-intervals}.
\begin{lemma}[Asymptotic estimate of $H_i+G_i$]\label{lem:h+g} Assume RH. Let $T\ge 4$, and let $0<\Delta= o(\log^2 T).$ Fix $0<\beta\le 1$, and choose $x=T^\beta$. Define the function $c(v)$ as in \eqref{eq:c-function}. Then, as $T\to\infty,$ we have
\begin{align*}
    &\textup{(a)} \ H_1+G_1=\frac{T}{2}\left\{\log \log T + \gamma_0 + 
    \sum_{m=2}^\infty\sum_{p\ge 2}\frac{1}{p^m}\left(\frac{1}{m^2}-\frac{1}{m}\right)
    +\log \beta  
    -\int\limits_0^1\alpha \,  g(\alpha)^2 \, \d \alpha  \right\}
    +O\!\left(\frac{T\log \log T}{\log T}\right);\\
    &\textup{(b)} \ H_2+G_2 =
    T\left\{
    \int\limits_{0}^{\Delta\beta\log T} \frac{1-\cos u}{u}\d u 
   +c\!\left(\Delta\right)  
   -\int\limits_0^1 \alpha \, [1-\cos(\Delta\beta\log T \alpha)]g^2(\alpha) \, \d \alpha
    \right\} + o(T),
\end{align*}
where the error term $o(T)$ in part (b) is actually \[O\!\left(\frac{T}{\sqrt{\log T}}\right)+O\!\left(\frac{T\Delta}{\log^2 T}\right).\]
\end{lemma}
\noindent \textit{Proof.} We split the proof into the following subsections.
\subsubsection{Proof of part (a)}
We add the results of Lemma \ref{lem:G} and Lemma \ref{lem:H} and use that, by Lemma \ref{lem:functions}, $u^2g(u)^2=(1-f(u))^2$. This yields
\begin{equation}\label{eq: HMinusGStep2}
G_1+H_1 =
\frac{T}{2} \sum\limits_{n \leq x}\frac{\Lambda^2(n)}{n \log^2 n} -\frac{T}{2\log^2 x} \sum\limits_{n \leq x}\frac{\Lambda^2(n)}{n}  g^2\!\left( \frac{\log n}{\log x} \right) + O\!\left(\frac{T\log \log T}{\log T}\right).
\end{equation}
For the first term, we separate the primes from the prime powers and use Merten's Theorem, which states
\begin{equation*}
    \sum_{p\le x}\frac{1}{p} = \log\log x + \gamma_0 -\sum_{m=2}^\infty \sum_{p\ge 2}\frac{1}{m\,p^m} +O\!\left(\frac{1}{\log x} \right).
\end{equation*}
Therefore, we see that
\begin{equation}\label{eq:h+g_term1}
    \sum\limits_{n \leq x}\frac{\Lambda^2(n)}{n \log^2 n} = \log \log x +\gamma_0 + \sum_{m=2}^\infty\sum_{p\ge 2}\frac{1}{p^m}\left(\frac{1}{m^2}-\frac{1}{m}\right)+O\!\left(\frac{1}{\log x} \right).
\end{equation}
For the second term, unconditionally, note that the prime number theorem with error term implies that
\begin{equation*}
    P(y):=\sum_{n\le y}\frac{\Lambda^2(n)}{n} = \frac{\log^2 y}{2} + O(1).
\end{equation*}
Then, using summation by parts and integration by parts, we obtain
\begin{equation}\label{eq:h+g_term2}
   \sum\limits_{n \leq x}\frac{\Lambda^2(n)}{n}  g^2\!\left( \frac{\log n}{\log x} \right)=  \log^2 x\int\limits_{0}^{1}  \alpha\, g^2\!\left(\alpha \right)\,  \d \alpha + O\!\left( \log x \right).
\end{equation}
By inserting  \eqref{eq:h+g_term1} and \eqref{eq:h+g_term2} into \eqref{eq: HMinusGStep2}, we obtain part (a).
\subsubsection{Proof of part (b): Summing by parts} Similarly, we have
\begin{align}\label{eq:g2h2_step1}
        G_2+H_2 =
&T \sum\limits_{n \leq x}\frac{\Lambda^2(n)}{n \log^2 n}\left[1-\cos\!\left(\Delta \log n \right)\right]\nonumber\\
&-\frac{T}{\log^2 x} \sum\limits_{n \leq x}\frac{\Lambda^2(n)}{n}  g^2\!\left( \frac{\log n}{\log x} \right)\left[1-\cos\!\left(\Delta \log n \right)\right] 
+ O\!\left(\frac{T\log \log T}{\log T}\right).
\end{align}
Using summation by parts, integration by parts, and the prime number theorem with error term, we obtain
\begin{align}
    \sum\limits_{n \leq x}\frac{\Lambda^2(n)}{n}  g^2\!\left( \frac{\log n}{\log x} \right)\left[1-\cos\!\left(\Delta \log n \right)\right] &= 
    \sum\limits_{n \leq x}\frac{\Lambda(n)\log n}{n}  g^2\!\left( \frac{\log n}{\log x} \right)\left[1-\cos\!\left(\Delta \log n \right)\right] +O(1)\nonumber\\
    &=\log^2 x\int\limits_0^1 \alpha\,[1-\cos(\Delta\beta\log T\alpha)]\,g^2(\alpha)\,\d \alpha + O\!\left(\Delta+1\right). \label{eq:h+g2_term2}
\end{align}
To estimate the first term on the right-hand side of \eqref{eq:g2h2_step1}, consider the quantity 
\begin{equation*}
    M(y):=\sum_{n\le y} \Lambda^2(n) = y\log y - y + E(y),
\end{equation*}
so that $\d M(y) = \log y \, \d y + \d E(y)$ and $E(y)=O_N\!\left(\frac{y}{\log^N y}\right)$ is defined in \eqref{eq:E_def}. 
For this, we let $1<\ell<2$ be a parameter. We anticipate that we will eventually take $\ell\to 1^+$.
Then, the sum in the first term of \eqref{eq:g2h2_step1} is 
\begin{align*}
    \sum_{n\le x} \frac{\Lambda^2(n)}{n\log^2 n}
    [1-\cos(\Delta \log n)] &= 
    \int\limits_{\ell}^{x^+}\frac{1-\cos(\Delta \log y)}{y\log^2 y}\d M(y) \\
    &= \int\limits_{\ell}^{x^+}\frac{1-\cos(\Delta \log y)}{y\log y}\d y + \int\limits_{\ell}^{x^+}\frac{1-\cos(\Delta \log y)}{y\log^2 y}\d E(y).
\end{align*}
We use the change of variables $u=\Delta\log y$ in the first integral. For the second integral, we integrate by parts and use that $E(\ell) = \ell-\ell\log \ell$ to find that
\begin{align}\label{eq:l12-12}
    T\sum_{n\le x} \frac{\Lambda^2(n)}{n\log^2 n}
    [1-\cos(\Delta \log n)] &=
    T\Bigg\{
    \int\limits_{\Delta\log \ell}^{\Delta\log x} \frac{1-\cos u}{u}\d u 
    +(1-\cos(\Delta \log  \ell)) \left( \frac{\log \ell -1}{\log^2 \ell}  \right) \nonumber\\
    &\,\,\,\,\,\,\,+ \int\limits_\ell^\infty \frac{E(y)}{y^2\log^3 y}[-\Delta \log y \sin(\Delta\log y) + (1-\cos(\Delta\log y))(\log y + 2)] \d y
    \Bigg\} \nonumber\\
    &\,\,\,\,\,\,\,+ O\!\left(\frac{(\Delta +1)T}{\log^{N+1} x} \right).
\end{align}
Here, we used that $E(y)\ll_N \dfrac{y}{\log^{N} y}$ (for any $N>0$) to extend the last integral to infinity, up to an error term.  Now, we let $\ell\to 1^+$. Note that 
\begin{equation*}
    \lim_{\ell\to 1} (1-\cos(\Delta \log  \ell)) \left( \frac{\log \ell -1}{\log^2 \ell}  \right) = -\frac{\Delta^2}{2}.
\end{equation*}
Additionally, since $E(y) = y-y\log y$ for all $1\le y <2$, the second integrand above satisfies, in this range, 
\begin{equation*}\label{eq:E_estimate}
    \frac{E(y)}{y^2\log^3 y}[-\Delta \log y \sin(\Delta\log y) + (1-\cos(\Delta\log y))(\log y + 2)] = \frac{\Delta^2}{2} + O\big(\Delta^2(y-1)\big).
\end{equation*}
This shows that the second integral is absolutely convergent on $(1,\infty).$ Therefore, recalling that 
$x=T^\beta$ and $\Delta\ll \log^2 T$, we may let $\ell\to 1^+$ in \eqref{eq:l12-12} to find that
\begin{equation}\label{eq:l12-13}
\begin{split}
    T\sum_{n\le x} \frac{\Lambda^2(n)}{n\log^2 n}
    \left[1-\cos\!\left(\Delta \log n \right)\right]&=
    T\left\{
    \int\limits_{0}^{\Delta\beta\log T} \frac{1-\cos u}{u}\d u +c\!\left(\Delta\right) \right\}
   + O\!\left(\frac{T}{\log T} \right),
\end{split}
\end{equation}
where $c(v)$ is defined in \eqref{eq:c-function}.
By combining \eqref{eq:g2h2_step1}, \eqref{eq:h+g2_term2}, and \eqref{eq:l12-13}, we complete the proof of Lemma \ref{lem:h+g}. We remark that the restriction $\Delta=o(\log^2 T)$ comes from the sum over primes in equation \eqref{eq:h+g2_term2}. 
\qed

\section{Proofs of main theorems} 
\indent We now explain how Theorems \ref{thm:variance}, \ref{thm:short-intervals-keating}, \ref{thm:short-intervals}, and \ref{thm:short-intervals-berry} follow from the combination of our previous lemmas.  
\subsection{Proof of Theorem \ref{thm:variance}} For all $T \in \{T_n\}$, the proof of Theorem \ref{thm:variance} follows from inputting part (a) of Lemmas \ref{lem:r} and \ref{lem:h+g} into the representation formula for $\log|\zeta(\frac{1}{2}+it)|$, which we proved in part (a) of Lemma \ref{lem:represFinal}. Some of the integrals in these results are over the interval $[1,T]$, but these can easily be extended to $[0,T]$ since 
\begin{equation*}
 \int\limits_0^1\log^2|\zeta(\tfrac{1}{2}+it)| \d t \ll 1.
\end{equation*}
In particular, Theorem \ref{thm:variance} holds for all $T \in \{T_n\}$ such that $T\ge 4$. We now extend this result to hold for all $T\ge 4$.

\indent Assume $T_n \le T \le  T_{n+1}$. Since the integrand in Theorem \ref{thm:variance} is positive, we know that 
\begin{equation*}
 \int\limits^{T_n}_{0}\log^2\left|\zeta\!\left( \tfrac{1}{2} +it\right)\right| \, \d t \le \int\limits^{T}_{0}\log^2\left|\zeta\!\left( \tfrac{1}{2} +it\right)\right| \, \d t \le \int\limits^{T_{n+1}}_{0}\log^2\left|\zeta\!\left( \tfrac{1}{2} +it\right)\right| \, \d t 
\end{equation*}
Moreover, because both $T_n$ and $T_{n+1}$ are at most $1$ away from $T$ and Theorem \ref{thm:variance} holds for $T_n$ and $T_{n+1}$, by part (a) of Lemma \ref{lem:a(T)}
it follows that 
\begin{equation*}
 \int\limits^{T}_{0}\log^2\left|\zeta\left( \tfrac{1}{2} +it\right)\right| \, \d t = \frac{T}{2}\log \log T + aT +o(T),
\end{equation*}
which completes the proof of Theorem \ref{thm:variance} for all $T\ge 4$.  
\subsection{Proof of Theorems \ref{thm:short-intervals} and \ref{thm:short-intervals-berry}} For the proof of Theorem \ref{thm:short-intervals}, when we input part (b) of Lemmas \ref{lem:r} and \ref{lem:h+g} into part (b) of Lemma \ref{lem:represFinal}, we get 
\begin{align}\label{MyMR2ProofStep1}
  \int\limits_1^T &\left[\log \left|\zeta\!\left(\tfrac{1}{2}+i t+i\Delta\right)\right|-\log \left|\zeta(\tfrac{1}{2}+i t)\right|\right]^2 \d t  \nonumber\\
    &=T\left\{   \int\limits_{0}^{\Delta\beta\log T} \frac{1-\cos u}{u}\d u -w(\Delta)\int\limits^{\Delta\log T}_{\Delta \beta\log T} \frac{\cos u}{u}\d u   - \log \beta \right.\nonumber\\
    &\quad \quad + \int\limits_{0}^{1} v \ g^2\left( v \right)\cos(\Delta v \beta\log T)\left(1- w(\Delta) \right) \d v \nonumber\\
    &\quad \quad+ c(\Delta)+  \left. \frac{1}{2}\int\limits_{1}^{\infty} \frac{2F(\alpha) -F_{\Delta}(\alpha)-F_{-\Delta}(\alpha)}{\alpha^2}\, \d \alpha \right\} + o(T).
\end{align}
Because our results hold independently of our choice of $\beta$, there should be no $\beta$ dependence in our final result. First, note that, by analyzing separately the cases $\Delta\ll 1$ and $\Delta \gg 1$ and using the definition of $w(u)$, we have
\begin{equation*}
    \frac{\left|1- w(\Delta)\right|}{\Delta\log T} \ll \frac{1}{\log T},
\end{equation*}
uniformly for $\Delta>0$. 
We use this fact and combine the first three terms on the right-hand side of \eqref{MyMR2ProofStep1} to yield
\begin{align*}\label{MyMR2ProofStep2}
& \int\limits_{0}^{\Delta\beta\log T} \frac{1-\cos u}{u}\d u -w(\Delta)\int\limits^{\Delta\log T}_{\Delta \beta\log T} \frac{\cos u}{u}\d u   - \log \beta \nonumber\\
& \,\,\,\,\,\,= \int\limits_{0}^{\Delta\log T} \frac{1-\cos u}{u}\d u -\left(w(\Delta)-1\right)\int\limits_{\Delta\beta\log T}^{\Delta\log T} \frac{ \cos u }{u} \d u \nonumber\\
& \,\,\,\,\,\,= \int\limits_{0}^{1} \frac{1-\cos (\Delta \alpha\log T) }{\alpha}\d \alpha +O\!\left(\frac{1}{\log T}\right),\nonumber
\end{align*}
where we used integration by parts in the last line.
Next we consider the integral involving $g^2(v)$ on the right-hand side of \eqref{MyMR2ProofStep1}. 
Using integration by parts, we similarly see that
\begin{equation*}
\begin{split}
\int\limits_{0}^{1} &v \ g^2\left( v \right)\cos(\Delta v \beta\log T)\left(1- w(\Delta) \right) \d v 
\ll \frac{1}{\log T}.
\end{split}
\end{equation*}
Combining these simplified expressions together gives
\begin{equation*}
\begin{split}
  \int\limits_1^T &\left[\log \big|\zeta \big(\tfrac{1}{2}+i t+i\Delta\big)\big|-\log \big|\zeta\big(\tfrac{1}{2}+i t\big)\big|\right]^2 \d t \\
    &=T\left\{  \int\limits_{0}^{1} \frac{1-\cos (\Delta \alpha\log T )}{\alpha}\, \d \alpha
    + \frac{1}{2}\int\limits_{1}^{\infty} \frac{2F(\alpha) -F_{\Delta}(\alpha) -F_{-\Delta}(\alpha)}{\alpha^2}\, \d \alpha 
   \right\} + T\,c\!\left(\Delta\right)
    + o(T).
\end{split}
\end{equation*}
We then extend the range of integration to $[0,T]$ since, by Theorem \ref{thm:variance} and the Cauchy-Schwarz inequality, we have
\begin{equation*}
\int\limits_0^1 \left[\log \big|\zeta \big(\tfrac{1}{2}+i t+i\Delta\big)\big|-\log \big|\zeta\big(\tfrac{1}{2}+i t\big)\big|\right]^2 \d t \ll (1+\Delta)\log \log (3+\Delta).
\end{equation*}
This completes the proof of Theorem \ref{thm:short-intervals} for $T\in \{T_n\}$. Since the integrand is non-negative, this result can be extended to all $T\ge 4$ using the same argument as in the proof of Theorem \ref{thm:variance} along with part (b) of Lemma \ref{lem:a(T)}. Finally, to prove Theorem \ref{thm:short-intervals-berry}, recall that \eqref{eq:l12-13} implies that
\begin{equation*}
    \begin{split}
    T\sum_{n\le T} \frac{\Lambda^2(n)}{n\log^2 n}
    \left[1-\cos\!\left(\Delta \log n \right)\right]&=
    T\left\{
    \int\limits_{0}^{1} \frac{1-\cos \Delta u\log T}{u}\d u +c\!\left(\Delta\right) \right\}
   + O\!\left(\frac{T}{\log T} \right),
\end{split}
\end{equation*}
where $c(v)$ is defined in \eqref{eq:c-function}. Therefore, Theorem \ref{thm:short-intervals-berry} is equivalent to Theorem \ref{thm:short-intervals}. 

\subsection{Proof of Theorem \ref{thm:short-intervals-keating}} To prove Theorem \ref{thm:short-intervals-keating}, we need to express the sum over primes in terms of the logarithmic derivative of $\zeta(s)$ near $s=1$.
\begin{lemma}\label{lem:perron}
Assume RH. Let $x\ge 2$ and $u\in i\R$ with $0<|u|\le \sqrt{x}$. Then, 
\begin{equation*}
    -\frac{\zeta'}{\zeta}(1+u) = \sum\limits_{n\le x}\frac{\Lambda(n)}{n^{1+u}} + \frac{x^{-u}}{u} + O\!\left(\frac{\log^2 x}{\sqrt{x}}\right),
\end{equation*}
where the implied constant is universal.
\end{lemma}
\begin{proof}
This can be established using classical arguments in a similar manner to the proof of the prime number theorem (assuming RH), e.g.~ \cite[Chapter 13]{MV2}.
\end{proof}
Clearly,
\begin{align}\label{eq:primes1}
    \sum_{n\le T} \frac{\Lambda^2(n)}{n\log^2 n} 
    \left(1-\cos\!\left(\Delta\log n\right)\right) = \widetilde{C}(\Delta) +
     \sum_{n\le T} \frac{\Lambda(n)}{n\log n} 
    \left(1-\cos\!\left(\Delta\log n\right)\right) + O\left(\frac{1}{T}\right),
\end{align}
where $\widetilde{C}(v)$ is defined in \eqref{eq:c-tilde-def}. Also, note that
\begin{equation*}
    \frac{1-\cos(\Delta \log n)}{\log n} = - \frac{1}{2}\int\limits_0^{i\Delta} \Big( n^{u}-n^{-u} \Big)\,\d u.
\end{equation*}
Therefore, Lemma \ref{lem:perron} yields
\begin{align}\label{eq:primes2}
    \sum\limits_{n\le T}\frac{\Lambda(n)}{n\log n} (1-\cos(\Delta \log n)) &= 
    -\int\limits_0^{i\Delta} \left(\sum\limits_{n\le T}\frac{\Lambda(n)}{n^{1-u}}-\sum\limits_{n\le T}\frac{\Lambda(n)}{n^{1+u}}\right)\, \d u \nonumber \\
    &= \int\limits_0^{i\Delta}\left(\frac{\zeta'}{\zeta}(1-u)-\frac{\zeta'}{\zeta}(1+u)-\frac{T^{-u}+T^{u}}{u}\right)\, \d u + O\left(\frac{\Delta\, \log^2 T}{\sqrt{T}}\right).
\end{align}
Theorem \ref{thm:short-intervals-keating} now follows from Theorem \ref{thm:short-intervals-berry}, \eqref{eq:primes1}, \eqref{eq:primes2}, and a change of variables.
\section{Transition between ranges}\label{sec:transition}
In this section, we prove Corollary \ref{cor:berry} in both ranges. We begin by showing how Theorem \ref{thm:short-intervals} reduces to Fujii's theorem in \eqref{FujiiResult} when $\Delta=o(1)$.
\begin{proposition}\label{prop:fujii-reduc} Assume RH. 
Let $T\ge 4$ and $\Delta = O(1).$ Then
\begin{equation*}
    \frac{1}{2}\int\limits_1^\infty \frac{2F(\alpha)-F_\Delta(\alpha)-F_{-\Delta}(\alpha)}{\alpha^2}\, \d \alpha
    = \int\limits_1^\infty \frac{F(\alpha)\,[1-\cos(\Delta \alpha\log T)]}{\alpha^2}\, \d \alpha + O\!\left(\Delta\right),
\end{equation*}
as $T\to\infty$.
\end{proposition}
\begin{proof}
Recall the identity \eqref{eq:F_delta-id-positive}:
\begin{equation*}
    2F(\alpha) -F_{\Delta}(\alpha)-F_{-\Delta}(\alpha) = 
\frac{8\pi^2}{T\log T} \int\limits_{-\infty}^\infty e^{-4\pi |u|}\left[1-\cos\!\left(\Delta \alpha\log T+2\pi\Delta u\right)\right] \left|\sum_{0<\gamma\le T} T^{i\alpha \gamma}e^{2\pi i u \gamma}
    \right|^2 \d u.
\end{equation*}
By the mean-value theorem, $\cos\!\left(\Delta \alpha\log T+2\pi\Delta u\right)=\cos(\Delta \alpha\log T)+O\!\left(\Delta |u|\right)$. We also have the identity
\begin{equation*}
    F(\alpha)  = 
\frac{4\pi^2}{T\log T} \int\limits_{-\infty}^\infty e^{-4\pi |u|} \left|\sum_{0<\gamma\le T} T^{i\alpha \gamma}e^{2\pi i u \gamma}
    \right|^2 \d u.
\end{equation*}
Therefore, we obtain 
\begin{align*}
        2F(\alpha) -F_{\Delta}(\alpha)-F_{-\Delta}(\alpha) = &F(\alpha)\,[1-\cos(\Delta \alpha\log T)]\nonumber\\
        &+O\!\left(\Delta \int\limits_{-\infty}^\infty e^{-4\pi |u|}|u| \left|\sum_{0<\gamma\le T} T^{i\alpha \gamma}e^{2\pi i u \gamma}
    \right|^2 \d u\right).
\end{align*}
The rest of the proof consists of controlling this last error term. This requires a technical but straightforward modification of Montgomery's arguments and definitions in \cite{Mo}, which we define and prove in the appendix. In particular, we define $\widetilde{F}_{\sigma_0}(\alpha)$ in \eqref{eq:F_sigma}, which is a modification of $F(\alpha)$ by using a slightly different weight. Using the estimate $|u|\ll e^{4\pi|u|\varepsilon}$ for $\varepsilon>0$ and the identity \eqref{eq:ws-hat} for $\widetilde{F}_{\sigma_0}(\alpha)$, we find that 
\begin{equation*}
    2F(\alpha) -F_{\Delta}(\alpha)-F_{-\Delta}(\alpha) = F(\alpha)\,[1-\cos(\Delta \alpha\log T)] + O\!\left(\Delta  \widetilde{F}_{\sigma_0}(\alpha)\right),
\end{equation*}
where $\sigma_0=1-\varepsilon$ (we may take any $0<\varepsilon<\frac{1}{2}$). Now, Proposition \ref{prop:F_sigma_average} implies that 
\[\int\limits_1^\infty \frac{ \widetilde{F}_{\sigma_0}(\alpha)}{\alpha^2}\, \d \alpha \ll 1.
\]
Hence the desired result now follows.
\end{proof}

\begin{proof}[Proof of Corollary \ref{cor:berry}]
Note that, by Proposition \ref{prop:fujii-reduc}, since Theorem \ref{thm:short-intervals} reduces to Fujii's theorem in \eqref{FujiiResult} when $\Delta=o(1)$, part (a) follows from Fujii's remarks \cite[Section 3]{Fu1}. For part (b), let $\Delta\gg 1$.
We want to show that 
\begin{equation*}
    \pi^2\int\limits_0^T \left[S\!\left(t+\Delta\right)-S(t)\right]^2 \d t = T\left[
    \sum\limits_{n \leq T}\frac{\Lambda^2(n)}{n \log^2 n}\left(1-\cos\!\left(\Delta \log n \right)\right)
    +1\right] + o(T). 
\end{equation*}
To prove part (b) of Conjecture \ref{con:berry}, by Theorem \ref{thm:short-intervals-berry}, it is enough to show that
\begin{equation*}
    \frac{1}{2}\int\limits_{1}^{\infty} \frac{2F(\alpha) -F_{\Delta}(\alpha) -F_{-\Delta}(\alpha)}{\alpha^2}\, \d \alpha =1+ o(1).
\end{equation*}
By Conjecture \ref{con:chan}, we have
\begin{equation*}
    \frac{1}{2}\int\limits_{1}^{\infty} \frac{2F(\alpha) -F_{\Delta}(\alpha) -F_{-\Delta}(\alpha)}{\alpha^2}\, \d \alpha = 
    \int\limits_{1}^{\infty} \frac{1-\cos\left(\Delta \alpha\log T \right)w\left(\Delta\right)}{\alpha^2}\, \d \alpha +o(1).
\end{equation*}
Now note that
\begin{equation*}
    \int\limits_1^\infty \frac{1}{\alpha^2}\,\d \alpha =1.
\end{equation*}
Then, integrating by parts, 
we find that 
\begin{equation*}
\begin{split}
    \int\limits_{1}^{\infty} \frac{\cos\!\left(\Delta \alpha\log T \right) w\!\left(\Delta\right)}{\alpha^2}\, \d \alpha = 
        O\!\left(\frac{1}{\Delta\log T}\right)
    = O\!\left(\frac{1}{\log T}\right) ,
    \end{split}
\end{equation*}
as we wanted. This completes the proof. 
\end{proof}

\section*{Appendix: variations of Montgomery's weight}
Assume RH. Let $\frac{1}{2}<\sigma_0<\frac{3}{2}$, and define 
\begin{equation}\label{eq:F_sigma}
    w_{\sigma_0}(u):= \frac{4\sigma_0^2}{4\sigma_0^2+u^2} \ \ \ \ \ \textrm{and}\ \ \ \ \  \widetilde{F}_{\sigma_0}(\alpha):=\frac{2\pi}{T\log T}\sum_{0<\gamma, \, \gamma'\le T}T^{i\alpha(\gamma-\gamma')}w_{\sigma_0}(\gamma-\gamma').
\end{equation}
Note that we recover Montgomery's function $F(\alpha)$ by taking $\sigma_0=1.$ Since
\begin{equation*}
    \widehat{w_{\sigma_0}}(y)=2\pi\sigma_0e^{-4\pi\sigma_0|y|},
\end{equation*} 
we have the identity
\begin{equation}\label{eq:ws-hat}
    \widetilde{F}_{\sigma_0}(\alpha) = \frac{4\pi^2\sigma_0}{T\log T}\int\limits_{-\infty}^\infty e^{-4\pi\sigma_0|y|}\left|\sum_{0<\gamma\le T}T^{i\alpha \gamma}e^{2\pi y\gamma}
    \right|^2\, \d y.
\end{equation}
In particular, $\widetilde{F}_{\sigma_0}(\alpha)\ge 0$, and $\widetilde{F}_{\sigma_0}$ is even. Following Montgomery \cite{Mo} (see also \cite{GoMo}), we have the following asymptotic formula for $\widetilde{F}_{\sigma_0}(\alpha)$.
\begin{proposition}\label{prop:mont-sigma}
Let $\tfrac{1}{2}<\sigma_0<\tfrac{3}{2}$, and define $\widetilde{F}_{\sigma_0}(\alpha)$ as in \eqref{eq:F_sigma}. We have
\begin{equation*}
    \widetilde{F}_{\sigma_0}(\alpha) = \sigma_0T^{-2|\alpha|\sigma_0}\log T(1+o(1)) +|\alpha| + o(1),
\end{equation*}
uniformly for $0\le |\alpha|\le 1$, as $T\to\infty$.
\end{proposition}
\begin{proof}
In Montgomery's explicit formula, we take $\sigma=\frac{1}{2}+\sigma_0$ to obtain, for any $\frac{1}{2}<\sigma_0<\frac{3}{2}$ and $x\ge 1$, 
\begin{align*}
    2\sigma_0 \sum_{\gamma}\frac{x^{i\gamma}}{\sigma_0^2 + (t-\gamma)^2} = 
    &-x^{-\sigma_0}\sum_{n\le x}\frac{\Lambda(n)n^{\sigma_0-1/2}}{n^{it}}
    -x^{\sigma_0}\sum_{n> x}\frac{\Lambda(n)}{n^{1/2+\sigma_0+it}}\\
    &+x^{-\sigma_0+it}(\log \tau +O(1)) + O(x^{1/2}\tau^{-1}),
\end{align*}
where $\tau=|t|+2$, and the implied constants depend only on $\sigma_0$ (which we henceforth assume to be fixed). We write the above as $L(x,\, t)=R(x,\, T)$. Note that 
\begin{equation*}
    \int\limits_{-\infty}^{\infty} \frac{1}{[\sigma_0^2+(t-\gamma)^2][\sigma_0^2+(t-\gamma')^2]}\, \d t = \frac{2\pi}{\sigma_0}\cdot \frac{1}{4\sigma_0^2+(\gamma-\gamma')^2} = \frac{2\pi}{4\sigma_0^3}\,w_{\sigma_0}(\gamma-\gamma').
\end{equation*}
Then, taking the absolute value, squaring, and integrating, following Montgomery's argument, we obtain
\begin{equation}\label{eq:ms-l}
\int\limits_{0}^T |L(x,\, T)|^2\, \d t = \frac{2\pi}{\sigma_0}\sum_{0<\gamma,\, \gamma'\le T}x^{i(\gamma-\gamma')}w_{\sigma_0}(\gamma-\gamma') +O(\log^3 T) = \frac{1}{\sigma_0}\widetilde{F}_{\sigma_0}(\alpha)\, T\log T + O(\log^3 T).
\end{equation}
Now, let us analyze $\int\limits_{0}^T |R(x,\, T)|^2$. For the Dirichlet series, using \cite[Corollary 3]{MV3}, we obtain
\begin{align}\label{eq:ms-dirich}
    \int\limits_0^T \left| 
-x^{-\sigma_0}\sum_{n\le x}\frac{\Lambda(n)n^{\sigma_0-1/2}}{n^{it}}
    -x^{\sigma_0}\sum_{n> x}\frac{\Lambda(n)}{n^{1/2+\sigma_0+it}}
\right|^2 \, \d t = 
&x^{-2\sigma_0}\sum_{n\le x}\frac{\Lambda^2(n)}{n^{1-2\sigma_0}}(T+O(n))  \nonumber\\
&+x^{2\sigma_0}\sum_{n>x} \frac{\Lambda^2(n)}{n^{1+2\sigma_0}}(T+O(n)).
\end{align}
Note that 
\[
\int\limits_{1}^x y^{2\sigma_0-1} \log y \, \d y = \frac{x^{2 \sigma_0} (2 \sigma_0 \log x-1)+1}{4 \sigma_0^2} \ \ \  \textrm{and} \ \ \  
\int\limits_{x}^\infty y^{-1 - 2 \sigma_0} \log y \, \d y =
\frac{x^{-2 \sigma_0} (2 \sigma_0 \log x+1)}{4 \sigma_0^2}.
\]
Then, by the prime number theorem with error term, \eqref{eq:ms-dirich} equals
\begin{equation*}
    \frac{T\log x}{\sigma_0} + O(T) + O(x\log x).
\end{equation*}
We note that on the left-hand side of \eqref{eq:ms-dirich} we may use an estimate of Goldston and Montgomery \cite[Lemma 7]{GoMo} instead of  \cite[Corollary 3]{MV3} to replace the error term $O(x\log x)$ with $O(T\sqrt{\log x})$. 
Continuing with our proof, we have 
\[
\int\limits_{0}^T\left|x^{-\sigma_0+it} (\log \tau +O(1))
\right|^2  \, \d t = \frac{T \log^2 T + O(T\log T)}{x^{2\sigma_0}}. 
\]
If we choose $x=T^\alpha$ for $0 \le \alpha \le 1-\varepsilon$, then following Montgomery's argument the above estimates imply that
\[R(T^\alpha,\, T) = T\log T\left(T^{-2\alpha \sigma_0} \log T\,(1+o(1)) +\frac{\alpha}{\sigma_0}+o(1)\right).
\]
We combine this with \eqref{eq:ms-l} to obtain the desired result for $|\alpha|\le 1-\varepsilon.$ As remarked above, by the argument of Goldston and Montgomery \cite[Lemma 7]{GoMo}, this can be extended uniformly to $|\alpha|\le 1.$
\end{proof}
We also note that the following estimate holds.
\begin{proposition}\label{prop:F_sigma_average}
Let $\tfrac{1}{2}<\sigma_0<\tfrac{3}{2}$, $\beta>1$,  and define $\widetilde{F}_{\sigma_0}(\alpha)$ as in \eqref{eq:F_sigma}. Then,
\begin{equation*}\label{eq:Fs-average}
    \int\limits_1^\beta \widetilde{F}_{\sigma_0}(\alpha)\, \d \alpha \ll \beta.
\end{equation*}
\end{proposition}
\begin{proof}
Using an argument of Goldston \cite[Lemma A]{Go}, 
this follows from Proposition \ref{prop:mont-sigma} and the fact that $\widetilde{F}_{\sigma_0}(\alpha)\ge 0$.
\end{proof}

\bigskip

\noindent{\sc Acknowledgements}. We thank Emanuel Carneiro, Tsz Ho Chan, Winston Heap, Jon Keating, and the anonymous referee for numerous helpful comments and suggestions. We also thank Dan Goldston for encouragement. MML was supported by fellowship from the Graduate School of the University of Mississippi. MBM~was supported by the NSF grant DMS-2101912 and the Simons Foundation (award 712898). EQ-H acknowledges support from CNPq (Brazil), the STEP Programme of ICTP (Italy), and the Austrian Science Fund (FWF) project P-35322.

\bibliographystyle{abbrv}
\bibliography{bibliography}
\end{document}